\documentclass[reqno]{amsart}

\usepackage{graphicx}
\usepackage{amsfonts,amsmath, amssymb}
\usepackage{marginnote}
\usepackage{color}
\usepackage{cite}
\usepackage{bm}

\numberwithin{equation}{section}

\theoremstyle{plain}
\newtheorem{theorem}{Theorem}[section]

\newtheorem{lemma}[theorem]{Lemma}
\newtheorem{proposition}[theorem]{Proposition}

\theoremstyle{remark}
\newtheorem{remark}[theorem]{Remark}

\def\k{\kappa}

\def\f{\frac}

\def\b{\bar}

\usepackage{tikz}

\newcommand{\dd}{{\rm d}}
\newcommand{\tn}{{\tilde{n}}}

\newcommand{\R}{\mathbb{R}}

\newcommand{\Z}{\mathbb{Z}}

\newcommand{\T}{\mathbb{T}}

\newcommand{\FA}{\bm{A}}
\newcommand{\FB}{\bm{B}}
\newcommand{\FE}{\bm{E}}

\newcommand{\FC}{\mathbf{C}}
\newcommand{\FD}{\mathbf{D}}
\newcommand{\Fe}{\mathbf{e}}
\newcommand{\FF}{\mathbf{F}}
\newcommand{\Ff}{\mathbf{f}}

\newcommand{\FH}{\mathbf{H}}

\newcommand{\FX}{\mathbf{X}}
\newcommand{\FR}{\mathbf{R}}

\newcommand{\FW}{\mathbf{W}}

\newcommand{\FU}{\mathbf{U}}

\newcommand{\FS}{\mathbf{S}}
\newcommand{\Fq}{\mathbf{q}}
\newcommand{\Fr}{\mathbf{r}}

\newcommand{\Fs}{\mathbf{s}}
\newcommand{\Ft}{\mathbf{t}}

\newcommand{\CB}{\mathcal{B}}

\newcommand{\CE}{\mathcal{E}}
\newcommand{\CF}{\mathcal{F}}

\newcommand{\CI}{\mathcal{I}}
\newcommand{\CJ}{\mathcal{J}}

\newcommand{\CQ}{\mathcal{Q}}
\newcommand{\CV}{\mathcal{V}}
\newcommand{\CG}{\mathcal{G}}
\newcommand{\CP}{\mathcal{P}}
\newcommand{\CX}{\mathcal{X}}

\newcommand{\hu}{\hat{u}_n}
\newcommand{\hv}{\hat{v}_n}
\newcommand{\hh}{\hat{h}_n}
\newcommand{\hg}{\hat{g}_n}

\newcommand{\pa}{\partial}
\newcommand{\ka}{\kappa}

\newcommand{\vep}{\varepsilon}

\begin{document}
	\date{\today}
\title[Stability of steady MHD boundary layers]{Validity of Prandtl expansions for steady MHD in the Sobolev framework}

\author[C.-J. Liu]{Cheng-Jie Liu}
\address[C.-J. Liu]{School of Mathematical Sciences, Institute of Natural Sciences, Center of Applied Mathematics, LSC-MOE and SHL-MAC,  Shanghai Jiao Tong University, Shanghai, China}
\email{liuchengjie@sjtu.edu.cn}

\author[T. Yang]{Tong Yang}
\address[T. Yang]{Department of Mathematics, City University of Hong Kong, Hong Kong;
\&  School of Mathematics and Statistics,   Chongqing University,  Chongqing, China}
\email{matyang@cityu.edu.hk}

\author[Z. Zhang]{Zhu Zhang}
\address[Z. Zhang]{Hong Kong Institute for Advanced Study, City University of Hong Kong, Hong Kong}
\email{zhuzhangpde@gmail.com}

\maketitle

\begin{abstract}
	This paper is concerned with the vanishing viscosity and magnetic resistivity limit for the two-dimensional steady incompressible MHD system on the half plane with no-slip boundary condition on velocity field and perfectly conducting wall condition on magnetic field. We prove the nonlinear stability of shear flows of Prandtl type with nondegenerate tangential magnetic field, but without any positivity or monotonicity assumption on the velocity field. It is in sharp contrast to the steady Navier-Stokes equations and reflects the stabilization effect of magnetic field. Unlike the unsteady MHD system, we manage the degeneracy on the boundary caused by no-slip boundary condition and obtain the estimates of solutions by introducing an intrinsic weight function and some good auxiliary functions.    
\end{abstract}

\tableofcontents

\section{Introduction}
In this paper, we consider the vanishing viscosity and magnetic resistivity limit of the two-dimensional steady MHD system in $\Omega=\{(x,y)~|~ x\in\T_\varrho, y>0\}:$
\begin{equation}
 \left\{
  \begin{aligned}\label{1.1}
         &\FU\cdot\nabla\FU+\nabla P-\FH\cdot\nabla\FH-\mu\vep\Delta\FU=\FF_\FU,\\
         &\FU\cdot\nabla\FH-\FH\cdot\nabla\FU-\ka\vep\Delta\FH=\FF_\FH,\\
         &\nabla\cdot \FU=\nabla\cdot \FH=0.
  \end{aligned}
 \right.
\end{equation}
Here $\FU=(u,v)$, $\FH=(h,g)$ and $P$ stand for the velocity field, magnetic field and total pressure respectively, and the vectors $\FF_\FU=(F_{1,\FU},F_{2,\FU}),\FF_{\FH}=(F_{1,\FH},F_{2,\FH})$ are given 
external forces. The tangential variable $x$ takes value in torus $\T_\varrho=\mathbb{R}/(2\pi \varrho)\mathbb{Z}$ with periodicity $2\pi \varrho$, and the normal variable $y>0$ with the boundary $\{y=0\}$. $\mu\vep$ and $\ka\vep$ are viscosity and magnetic resistivity coefficients respectively with $\vep\ll1$ and positive constants $\mu, \ka$. We impose the steady MHD system with the following no-slip boundary condition on velocity field and perfectly conducting wall condition on magnetic field:
\begin{align}\label{1.1-1}
\FU|_{y=0}=(\partial_y h,g)|_{y=0}=\mathbf{0}.
\end{align}
Moreover, it is natural to assume the compatibility condition for $\FF_{\FH}=(F_{1,\FH},F_{2,\FH})$:
\begin{align}\label{com-div}
	\nabla\cdot\FF_\FH=0,\quad F_{2,\FH}|_{y=0}=0.
\end{align} 

We are concerned with the asymptotic behavior of solutions $(\FU,\FH)$ to \eqref{1.1}-\eqref{1.1-1} as $\vep\rightarrow0$, and it is a high Reynolds numbers limit problem, one of the fundamental topics in hydrodynamics. It is well-known that such problems are very important and challenging in the presence of boundary, especially when considering the no-slip boundary conditions for the velocity field. The key issue is the large vorticity for small viscosity near the boundary, the so-called boundary layer phenomenon. A major and powerful tool for studying these problems is boundary layer theory, which was introduced by Prandtl \cite{P} in 1904. According to Prandtl's theory, the boundary layer of the system \eqref{1.1}-\eqref{1.1-1} is characteristic with the scale $\sqrt{\vep}$, and the solutions should have the following asymptotic behavior:
\begin{itemize}
	\item $(\FU,\FH)(x,y)\sim (\FU^I,\FH^I)(x,y)$ away from the boundary, where $(\FU^I,\FH^I)$ satisfies the ideal MHD system with the boundary conditions $\FU^I\cdot\vec{n}|_{y=0}=\FH^I\cdot\vec{n}|_{y=0}=0.$
	\item $(\FU,\FH)(x,y)\sim \left(u^p(x,\frac{y}{\sqrt{\vep}}),\sqrt{\vep}v^p(x,\frac{y}{\sqrt{\vep}}),h^p(x,\frac{y}{\sqrt{\vep}}),\sqrt{\vep}g^p(x,\frac{y}{\sqrt{\vep}})\right)$ near the boundary, where $(u^p,v^p,$
	$h^p,g^p)(x,Y)$ satisfies a Prandtl-type system with the boundary conditions\\ $(u^p,v^p,\partial_Yh^p,g^p)|_{y=0}=\mathbf{0}$ and far-field conditions: $\lim\limits_{Y\rightarrow+\infty}(u^p,h^p)(x,Y)$ matches  the trace of tangential components of $(\FU^I,\FH^I)$ on the boundary $\{y=0\}$. 
\end{itemize}
Mathematically, it follows two fundamental problems:
\begin{itemize}
	\item the well-posedness/ill-posedness of the Prandtl boundary layer system;
	\item the rigorous justification of the Prandtl expansion for small viscosity.
\end{itemize}
Let us stress that the first problem is of course very important and has 
{a lot of }results, while in the present paper we mainly focus on the second problem corresponding to no-slip boundary conditions for the flow. For more on boundary layer theory, see the reviews \cite{E,M-M} and the references therein.

Before stating the main result of this paper, we review some mathematical results on the validity of Prandtl asymptotics. Let us first focus on the situations related to the classical unsteady Navier-Stokes equations with no-slip boundary conditions. To the best of our knowledge, the first rigorous verification of the Prandtl boundary layer theory was achieved in the analytic framework for both 2D and 3D cases by Sammartino and Caflisch in the celebrated paper \cite{SC2}. One can also refer to \cite{WWZ} for a new proof for 2D case based on direct energy method. After that, a notable step forward in this direction was made by Maekawa, and in \cite{Mae2} he justified rigorously the Prandtl ansatz in the inviscid limit for 2D Navier-Stokes equations with the initial vorticity supported away from the boundary, which implies some kind of analyticity near the boundary. This result was generalized to 3D in \cite{FTZ}. Of course, it is more important to show the justification of Prandtl ansatz for data with finite Sobolev regularity, since it is more physically relevant. However, known results in this direction are still far from optimistic due to a number of reasons, such as the reverse flow, Tollmien–Schlichting wave and so on, see 
{physical} literatures \cite{D-R,Sch}. 
At the mathematical level, the existing results related to validity of Prandtl boundary layer theory in the Sobolev framework are also far from satisfactory, according to the instability of Prandtl asymptotics of shear flow type obtained in some recent papers. Precisely, Grenier and Nguyen established counterexamples to nonlinear stability of Prandtl boundary layer profiles with inflexion points in \cite{Grenier1, G-N,G-N2}. Even for the monotonic and concave Prandtl boundary layer profiles, we may not expect the nonlinear stability of Prandtl boundary layer in Sobolev setting. In the notable work \cite{GGN3}, the authors studied the linearized Navier–Stokes equations around generic stationary shear flows of the boundary layer type and constructed solutions with highly growing eigenmodes like $e^{t/{\nu^{\frac{1}{4}}}}$ ($\nu:$ viscosity) related to the $O(\nu^{-\frac{3}{8}})$ tangential frequency, see \cite{GGN2} for related statements and \cite{G-N1, G-N3,G-N4} for new progress. The result in \cite{GGN3} suggests somehow that one can only prove the validity of Prandtl boundary layer theory 
in the function spaces of Gevrey class, and recently there are several interesting  work in this direction, see \cite{CWZ,GMM1,GMM2}. 

Now we turn  to the steady Navier-Stokes case, and surprisingly the situation is more satisfactory than the unsteady case. The first rigorous result on the validity of steady Prandtl boundary layer profiles was proved by Guo and Nguyen in \cite{GN}, in which they consider the steady Navier-Stokes equations in the domain $\{(x,y)\in [0,L]\times \R_+\}$ with a positive Dirichlet boundary condition for the tangential velocity, the so-called moving plate. They constructed general boundary layer expansions for small viscosity and proved their validity in the Sobolev framework for small $L$, see also some generalizations \cite{Iy,Iy1,Iy2,Iy3,Iy4}. Note that the moving plate condition is not the no-slip boundary condition and avoids some difficulties from degeneracy on the boundary due to the vanishing tangential velocity. 
In \cite{G-I}, Guo and Iyer generalized the result to the case with homogeneous Dirichlet boundary conditions, the same as no-slip boundary condition. And the boundary layer profiles in the Prandtl ansatz studied in \cite{G-I} involve the famous Blasius flow. Very recently, Gao and Zhang gave a simplified proof of this result in \cite{G-Z}. In another important work \cite{GM}, G\'erard-Varet and Maekawa studied the steady Navier-Stokes equations with no-slip boundary condition and some additional source terms in the same domain as the present paper, and obtained the $H^1$ stability of shear flows of Prandtl type.

Back to the MHD system, its boundary layer theory is richer because of different choices of magnetic 
physical parameters, one can refer \cite{G-P,WX} for more details. In the 2D unsteady MHD system when viscosity and magnetic resistivity tend to zero at the same rate, the stabilization effect from non-degenerate tangential magnetic field was discovered in \cite{G-P,LXY1,LXY2}, and the validity of Prandtl boundary layer theory was rigorously proved in \cite{LXY2}, in sharp contrast with the unsteady Navier-Stokes system. 
As a further step in this direction, 
the purpose of this paper is to reveal the stability mechanism of magnetic field for the steady system \eqref{1.1}-\eqref{1.1-1} in order to justify the stability of shear flows of Prandtl type in the Sobolev framework. 
Let us mention that in  \cite{D-L-X} the authors extended the result in \cite{GN} to 2D steady MHD system with moving plate condition, and the stability mechanism is from the non-degenerate velocity field but
not from the magnetic field that is consistent with \cite{GN},   

To state the main result in this paper, let us first introduce some notations and assumptions. Denote by 
$$(\FU_s,\FH_s)(Y)=\big(U_s(Y),0,H_s(Y),0\big),\qquad Y:=\frac{y}{\sqrt{\vep}}$$ 
 a background shear flow with $U_s(0)=H_s'(0)=0$. 
We can see that it is a special solution to \eqref{1.1} when the external forces
$$(\FF_{\FU},\FF_{\FH})=(\FF_{\FU_s},\FF_{\FH_s}):=(-\vep\mu\pa_y^2U_s,0,-\vep\ka\pa_y^2H_s,0)=\left(-\mu\pa_Y^2U_s(Y),0,-\ka\pa_Y^2H_s(Y),0\right).$$ 
We are interested in a general class of shear flow that satisfies the following assumptions.\\

{\bf Assumptions:}
\begin{itemize}
	\item $U_s,H_s\in C^3(\mathbb{R}_+)\cap C^1(\overline{\mathbb{R}_+})$ such that 
	\begin{align}\label{A1}
	U_s(0)=0,~ H_s'(0)=0,~\lim_{Y\rightarrow +\infty}U_s(Y)=U_E, ~\lim_{Y\rightarrow +\infty}H_s(Y)=
	H_E\neq 0,
	\end{align}
	and
	\begin{align}\label{A3}
		\bar{M}:=\sum_{1\leq k\leq3}\sup_{Y\geq0}(1+Y)^3\left(|\pa_Y^kU_s(Y)|+|\pa_Y^kH_s(Y)|\right)<\infty.
	\end{align}
	\item There are two positive constants $\underline{\gamma},\bar{\gamma}>0$, such that 
	\begin{align}
	\underline{\gamma}\leq |H_s(Y)|\leq \bar{\gamma},\text{ for any }Y>0.\label{A2}
	\end{align}
	And set $\displaystyle G_s(Y):=H_s^2(Y)-U_s^2(Y),$ it holds 
	\begin{align}\label{A4}
		\gamma_0:=\inf_{Y\geq 0}G_s(Y)=\inf_{Y\geq 0}\left(H_s^2(Y)-U_s^2(Y)\right)>0.
	\end{align} 
\end{itemize}

Note that 
$\b{M}$  measures the amplitude of perturbation of the boundary layer profile $(U_s,H_s)$ around its far-field $(U_E,H_E)$, and \eqref{A4} implies that magnetic field dominates velocity field.

In this paper, we will show the stability of $(\FU_s,\FH_s)$ in Sobolev spaces for the problem \eqref{1.1}-\eqref{1.1-1}. Set $$(\widetilde{\FU},\widetilde{\FH})=(\FU,\FH)-(\FU_s,\FH_s)\triangleq(\tilde{u},\tilde{v},\tilde{h},\tilde{g})$$ 
be the perturbation of $(\FU_s,\FH_s)$. 
From \eqref{1.1}-\eqref{1.1-1} the problem for $(\widetilde{\FU},\widetilde{\FH})$ is written as
\begin{equation}\label{1.2}
\left\{
\begin{aligned}
&U_s\pa_x\widetilde{\FU}+\tilde{v}\pa_yU_s\Fe_1-H_s\pa_x\widetilde{\FH}-\tilde{g}\pa_yH_s\Fe_1+ \nabla P-\mu\vep\Delta\widetilde{\FU}=-\widetilde{\FU}\cdot\nabla \widetilde{\FU}+\widetilde{\FH}\cdot\nabla \widetilde{\FH}+\Ff_\FU,\\
&U_s\pa_x\widetilde{\FH}+\tilde{v}\pa_yH_s\Fe_1-H_s\pa_x\widetilde{\FU}-\tilde{g}\pa_yU_s\Fe_1-\ka\vep\Delta\widetilde{\FH}=-\widetilde{\FU}\cdot\nabla \widetilde{\FH}+\widetilde{\FH}\cdot\nabla \widetilde{\FU}+\Ff_\FH,\\
&\nabla\cdot \widetilde{\FU}=\nabla\cdot \widetilde{\FH}=0,\\
&\widetilde{\FU}|_{y=0}=(\partial_y\tilde{h},\tilde{g})|_{y=0}=\mathbf{0},
\end{aligned}
\right.
\end{equation}
where the vector $\Fe_1=(1,0)$, and the source term
$$(\Ff_\FU,\Ff_{\FH}):=(\FF_{\FU},\FF_{\FH})-(\FF_{\FU_s},\FF_{\FH_s})\triangleq(f_{1,\FU},f_{2,\FU},f_{1,\FH},f_{2,\FH})$$ 
satisfying $\nabla\cdot\Ff_\FH=0,~ f_{2,\FH}|_{y=0}=0$ by virtue of \eqref{com-div}. Before stating the main result, we  introduce the function spaces used in the paper. For any $x$-dependent function $f(x)\in L^2(\T_\varrho),$ we denote by $f_n$ its $n$-th Fourier coefficient, i.e.,
$$f_n=\frac{1}{2\pi\varrho}\int_0^{2\pi\varrho}e^{-i\tilde{n}x}f(x)\dd x,\qquad n\in\Z,\quad \tilde{n}=\f{n}{\varrho},
$$
and by $\CP_nf=f_ne^{i\tilde{n}x}$ the corresponding orthogonal projection on the $n$-th Fourier mode. The divergence-free and boundary conditions in \eqref{1.2} imply
\begin{align*}
	(\widetilde{\FU}_0,\widetilde{\FH}_0)=(\tilde{u}_0,0,\tilde{h}_0,0).
\end{align*}
We further denote $\CQ_0f=(I-\CP_0)f$ to be the projection on the non-zero Fourier modes. To study \eqref{1.2}, we need use a suitable solution space. Denote by $H^s$ and $\dot{H}^s, s\in\R$ the inhomogeneous and homogeneous Sobolev spaces respectively, and define the subspace of $H^s$:
\begin{align*}	H^s_\sigma:=\big\{\FU=(U_1,U_2)\in H^s(\Omega)~\big|~\nabla\cdot\FU=0,~U_2|_{y=0}=0
\big\}.\end{align*} 
Motivated by \cite{GM}, we define the function space $\CX$ for  $(\widetilde{\FU},\widetilde{\FH})=(\tilde{u},\tilde{v},\tilde{h},\tilde{g})$ as follows.
\begin{align}\label{X}
\begin{aligned}
\mathcal{X}=\bigg\{(\widetilde{\FU},\widetilde{\FH})~\big|~&
 \widetilde{\FU}|_{y=0}=(\pa_y\tilde{h},\tilde{g})|_{y=0}=\mathbf{0},\quad\left(\widetilde{\FU}_0,\widetilde{\FH}_0\right)=(\tilde{u}_0,0,\tilde{h}_0,0)\in L^\infty(\mathbb{R}_+)\cap \dot{H}^1(\mathbb{R}_+), \\
&\left(\CQ_0\widetilde{\FU},\CQ_0\widetilde{\FH}\right)\in H^{1}_\sigma(\Omega),	\quad \|(\widetilde{\FU},\widetilde{\FH})\|_{\CX}<\infty \bigg\},
\end{aligned}\end{align}	
where the norm $\|(\widetilde{\FU},\widetilde{\FH})\|_{\CX}$ is given by
\begin{align}\label{1.1-3}
\begin{aligned}
\|(\widetilde{\FU},\widetilde{\FH})\|_{\CX}:=&\sum_{n}\|(\widetilde{\FU}_n,\widetilde{\FH}_n)\|_{L^\infty(\mathbb{R}_+)}+\vep^{\f14}\|(\pa_y\tilde{u}_0,\pa_y\tilde{h}_0)\|_{L^2(\mathbb{R}_+)}+\|Z^{\f12}(\pa_y\tilde{u}_0,\pa_y\tilde{h}_0)\|_{L^2(\mathbb{R}_+)}\\
&+\vep^{-\f14}\|(\CQ_0\widetilde{\FU},\CQ_0\widetilde{\FH})\|_{L^2(\Omega)}+\vep^{-\f12}\|Z^{\f12}(\CQ_0\widetilde{\FU},\CQ_0\widetilde{\FH})\|_{L^2(\Omega)}\\
&+\vep^{\f14}\|(\nabla\CQ_0\widetilde{\FU},\nabla\CQ_0\widetilde{\FH})\|_{L^2(\Omega)}+\|Z^{\f12}(\nabla\CQ_0\widetilde{\FU},\nabla\CQ_0\widetilde{\FH})\|_{L^2(\Omega)}.
\end{aligned}\end{align}
Here 
the weight function $Z^{\frac{1}{2}}$ satisfies that $Z=Z(y)\in C^2(\R_+)$, $Z(y)\sim y$ for $y\in(0,2)$ and remains constant for $y\geq 2$. We will specify it later in Section 2. Moreover for simplicity we assume that $\Ff_\FU=\CQ_0\Ff_\FU$ and $\Ff_\FH=\CQ_0\Ff_{\FH}$, since one can extend our result to the general case by adding some shear flow profile, corresponding to the nonzero $\big(\Ff_{\FU,0},\Ff_{\FH,0}\big)$, to the solution $\big(\widetilde{\FU},\widetilde{\FH}\big)$. 

Our main result is presented as follows.
\begin{theorem} \label{thm1}
	Let $(\FU_s,\FH_s)$ be a given shear flow that satisfies assumptions \eqref{A1}-\eqref{A4}. There exist positive constants $\delta_1, \delta_2$ and $\vep_0$, such that for any $\vep\in(0,\vep_0)$ and  $\eta>0$, if 
	\begin{align}
	\varrho(\b{M}+\bar{M}^4)\in(0,\delta_1),\label{1.1-2}
 \end{align}
and $$\|(\Ff_\FU,\Ff_{\FH})\|_{L^2(\Omega)}+\vep^{-\f14}\|Z^{\f12}(\Ff_\FU,\Ff_{\FH})\|_{L^2(\Omega)}\leq \frac{\delta_2\vep^{\f34}}{|\log\vep|^{3+\eta}},$$ then \eqref{1.2} admits a unique solution $(\widetilde{\FU},\widetilde{\FH},\nabla P): (\widetilde{\FU},\widetilde{\FH})\in \CX\cap H^2_{loc}(\Omega), \nabla P\in L^2(\Omega)$ that satisfies the estimate: 
\begin{align}
\|(\widetilde{\FU},\widetilde{\FH})\|_{\CX}\leq C\vep^{-\f14}|\log\vep|^{\frac{3+\eta}{2}}\left(\|(\Ff_\FU,\Ff_{\FH})\|_{L^2(\Omega)}+\vep^{-\f14}\|Z^{\f12}(\Ff_\FU,\Ff_{\FH})\|_{L^2(\Omega)}\right),\label{est}
\end{align}
where $C$ is independent of $\varepsilon$. 
\end{theorem}

\begin{remark} Compared our work with the result in \cite{GM} for steady Navier-Stokes equations, there are three main differences.
	\begin{itemize}
		\item[(a)] The shear flow in \cite{GM} is monotonic near the boundary and remains positive for all $Y>0$. These assumptions are crucial for the stability since they prevent the reverse flow and boundary layer separation. While there is neither monotonicity nor positivity assumption on 
		the velocity field background in our result, instead the only structural condition we need is \eqref{A4}. 
		It reflects the stabilization effect of tangential magnetic field on the boundary layer.
		
		\item[(b)] Another essential  requirement for the stability result in \cite{GM} is the smallness condition on the periodicity $\varrho$ of the tangential variable, which
		means that the stability result is only local in space. However such smallness for 
		$\varrho$ is not necessary in our result. In some sense we have proven the almost global stability for $(\FU_{s},\FH_s)$. In fact, one can recover from 
		\eqref{1.1-2} that, the periodicity $\varrho$ can be arbitrarily large, provided that $\b{M}$, which measures the perturbations of the profile $(U_s,H_s)$ around its far field $(U_E,H_E)$, is suitably small.
		
		\item[(c)] Our analysis is quite different from \cite{GM}. As the authors \cite{GM} mentioned
		in their paper that  they are not able to get direct estimates of the perturbation $\widetilde{\FU}$, instead they
		construct the solution 
		{to} the problem of $\widetilde{\FU}$  via a complicated iteration process. However, by using a key transform inspired by \cite{LXY2} we can establish the estimates of $\widetilde{\FU}$ through a direct energy method. 
	\end{itemize}
\end{remark}

\begin{remark}
	We stress that the result in the present paper is a  generalization of \cite{LXY2} to the steady case. Unlike the unsteady case, for steady case it is difficult to establish the $L^2$ estimates for the equations \eqref{1.1}, which are degenerate on the boundary because of no-slip condition for velocity. It leads to some  essential difficulties in the mathematical analysis.   
\end{remark}

In what follows, we briefly  point out the difficulties and explain the main ingredients in our proof.
	\begin{itemize}
		\item[(a)] \underline{Good unknown functions}. First, to prove Theorem \ref{thm1}, the key step is to analyze the linear system \eqref{2.0}. Similar as the Navier-Stokes equations \cite{GM}, one of the difficulties in the analysis of \eqref{2.0} comes from the large stretching terms $v\pa_yU_s-g\pa_yH_s$ and $v\pa_yH_s-g\pa_yU_s$ which behave like $O(\vep^{-\f12})(v,g)$. As in \cite{LXY2}, our strategy to overcome this difficulty is to introduce  new unknowns $(\widehat{\FU},\widehat{\FH})=(\hat{u},\hat{v},\hat{h},\hat{g})$ that are defined in Section 3.2, in which the non-degeneracy of tangential magnetic field \eqref{A2} plays an important role. Notice that the transformation performed in the present paper is slightly different from that in \cite{LXY2}, since it keeps the divergence-free condition for both velocity field and magnetic field which is important for proof  in this paper. By reformulating \eqref{2.0} into a system for these new unknowns, the previously mentioned stretching terms are directly cancelled, see \eqref{pr-new}.  
		\item[(b)] \underline{$L^2$-coercivity}. The good unknown functions provide  an advantage  to obtain uniform-in-$\vep$ estimates of $\vep^{\frac{1}{2}}\|\nabla(\widehat{\FU},\widehat{\FH})\|_{L^2}$ via $\|(\widehat{\FU},\widehat{\FH})\|_{L^2}$, see Lemma \ref{prop2.3}. Then it remains  to establish the estimate of $\|(\widehat{\FU},\widehat{\FH})\|_{L^2}$ to make the process self-contained. 
		However, in contrast to the previous work \cite{LXY2} for the unsteady case, it is hard to obtain the $L^2$ estimate directly. Moreover, there is a difficulty  from  the degeneracy due to the no-slip boundary condition. Therefore, 
		another  key ingredient in the proof is to establish an $L^2$-coercivity estimate of linearized steady MHD operator around the boundary layer profile. To illustrate the main idea, let us consider the main part of \eqref{pr-new-n} for the $n$-th Fourier mode of  the good unknown function $(\widehat{\FU},\widehat{\FH})$: 
		$$\begin{aligned}
		-i\tn G_s(\frac{y}{\sqrt{\varepsilon}})~\widehat{\FH}_n+(i\tn p_n,\pa_yp_n)-\vep\mu(\pa_y^2-\tn^2)\widehat{\FU}_n&=\cdots,\\
		-i\tn \widehat{\FU}_n-\vep\ka(\pa_y^2-\tilde{n}^2)\widehat{\FH}_n&=\cdots,
		\end{aligned}
		$$
		where $G_s(Y)=H_s^2(Y)-U_s^2(Y)$. 
		Thanks to the non-degeneracy assumption \eqref{A4}, 
		$G_s$ has a strictly positive lower bound. A natural multiplier  is $(\widehat{\FH}_n,\widehat{\FU}_n)$
		to obtain  the estimate of $|\tn|^{\f12}\|(\widehat{\FU}_n,\widehat{\FH}_n)\|_{L^2}$. However, such a multiplier is not compatible with the diffusion terms of $\widehat{\FU}_n$ because the boundary term $\hat{h}_n\pa_y\hat{u}_n|_{y=0}$ appears due to the mixed boundary condition \eqref{1.1-1}.
		And this boundary term is clearly hard to control for this degenerate system.
		
		For this, we will establish a weighted estimate of the solution with an appropriate weight function $Z^{\frac{1}{2}}(y)$ which vanishes on the boundary, see Lemma \ref{prop2.5}. Then the interpolation inequality \eqref{AP1} allows us to obtain the estimate of $\|(\widehat{\FU},\widehat{\FH})\|_{L^2}$. In this process, since the un-weighted estimates and the weighted estimates are strongly coupled, we must keep track of the dependence of the constants on the frequency $n$, the length of torus $\varrho$ and $\bar{M}$ in each step. The smallness assumption in \eqref{1.1-2} is crucial for closing the estimate in $\CX$.
		\item[(c)] \underline{Choice of weight function}. The key issue in Lemma \ref{prop2.3} is to obtain
		 a gain of $\vep^{\f14}$ in the weighted estimate of magnetic field, which is crucial to recover the un-weighted $L^2$-estimate via the interpolation inequality \eqref{AP1}. Therefore, the hypothesis of $Z(y)\sim y$ near the boundary $\{y=0\}$ is natural, since multiplying terms involving $Y$-derivatives of boundary layer profile by the weight $Z^{\frac{1}{2}}(y)$ yields a gain of $\vep^{\f14}$. The main reason for the tricky construction of $Z(y)$ in Section 2 is as follows. We consider the vorticity formulation \eqref{2.11} (to avoid the commutator $[Z,\pa_yP_n]$). Thanks to the divergence-free condition for the good unknown function $(\hat{h},\hat{g})$, denote by $\hat{\psi}_n$ the $n$-th Fourier coefficient of the stream function $\hat{\psi}$ of $(\hat{h},\hat{g})$. Applying the multiplier $Z\hat{\psi}_n$ to the vorticity equation produces the following good terms:
		$$
		\begin{aligned}
		&\text{Im}\langle -i\tn [G_s(Y)\omega_{h,n}+\pa_yG_s(Y)\hat{h}_n],Z\hat{\psi}_n\rangle\\
		&=\int_0^\infty \tilde{n}G_s(Y)Z(y)|\widehat{\FH}_n|^2\dd y+\text{Im}\int_0^\infty i\tn\pa_yZ~ G_s(Y)\pa_y\hat{\psi}_n\overline{\hat{\psi}_n}\dd y\\
		&=\underbrace{\tn\int_0^\infty G_s(Y)Z(y)|\widehat{\FH}_n|^2\dd
			y}_{\CJ_1}+\underbrace{\frac{\tn}{2}\int_0^\infty\frac{-\pa_y[G_s(Y)\pa_yZ]}{2}|\hat{\psi}_n|^2\dd y}_{\CJ_2}.
		\end{aligned}
		$$
		Here $\displaystyle Y=\frac{y}{\sqrt{\varepsilon}}$, and $\CJ_1$ gives the desired weighted boundedness on $\widehat{\FH}_n$. So the function $Z(y)$ is designed so that the most singular part in the lower-order term $\CJ_2$ is cancelled. See Lemma \ref{lmz} for necessary details. Notice that such a process is not appropriate for the vorticity equation of magnetic field in \eqref{2.11}, simply because it is not strictly concave near the boundary. Fortunately, due to the boundary condition $\widetilde{\FU}|_{y=0}=\mathbf0$, the un-weighted norm $\|\widehat{\FU}_n\|_{L^2}$ can be obtained directly by applying the natural multiplier $\widehat{\FU}_n$ to the second equality in \eqref{pr-new-n}, see Lemma \ref{prop2.4} below.
		\item[(d)] \underline{Commutator estimates}.
		Note that the commutator $[\pa_y,Z]=\pa_yZ$ is not a boundary layer term, then the  gain of $\vep^{\f14}$ does not apply to it. Therefore,
		another key point in our weighted estimate is to control the lower order terms involving this commutator in a suitable way. To this end, we take as an example an inner product term $\langle R_n,\pa_yZ \pa_y^{-1}\hat{h}_n\rangle$ where $R_n$ is an inhomogeneous source. We observe that $\partial_yZ$ is supported on $[0,2]$, and the integral operator $\pa_y^{-1}$ gives an extra $Z(y)$ near the boundary, then it implies a trivial bound of this term by virtue of the Hardy inequality as
		\begin{align*}
		|\langle R_n,\pa_yZ \pa_y^{-1}\hat{h}_n\rangle|&=\left|\left\langle yR_n,\pa_yZ\frac{\pa_y^{-1}\hat{h}_n}{y}\right\rangle\right|\leq C\|ZR_n\|_{L^2}\|\hat{h}_n\|_{L^2}\\
		&\leq C\big(\vep^{-\f14}\|Z^{\f12}R_n\|_{L^2}\big)\cdot\big( \vep^{\f14}\|h_n\|_{L^2}\big),
		\end{align*}
		which will lead to a growth of $\vep^{-\f14}$ in our linear estimate \eqref{prop2-1}. If so, an extra $\vep^{\f12}$ on the perturbation of external force $(\Ff_{\FU},\Ff_{\FH})$ is required to compensate such a growth in the nonlinear analysis. In order to minimize the  negative power of $\vep$, our main idea is as follows. First, we use weighted Hardy inequality,  instead of the classical one, in the above treatment: 
		\begin{align*}
		|\langle R_n,\pa_yZ \pa_y^{-1}\hat{h}_n\rangle|&=\left|\left\langle \sqrt{y}R_n,\pa_yZ\frac{\pa_y^{-1}\hat{h}_n}{\sqrt{y}}\right\rangle\right|\leq C\big\|Z^{\frac{1}{2}}R_n\big\|_{L^2}\big\|Z^{\frac{1}{2}}|\log{Z} |^{1+}~\hat{h}_n\big\|_{L^2}.
		\end{align*}  
		 We emphasize that the logarithmic type weight in $\big\|Z^{\frac{1}{2}}|\log{Z} |^{1+}~\hat{h}_n\big\|_{L^2}$ is necessary since it is the critical case for Hardy inequality, see Lemma \ref{lmw1} for details. Second, we establish the control of $\big\|Z^{\frac{1}{2}}|\log{Z} |^{1+}~\hat{h}_n\big\|_{L^2}$ via $\big\|Z^{\frac{1}{2}}\hat{h}_n\big\|_{L^2}$ and $\varepsilon^{\frac{1}{4}}\big\|\hat{h}_n\big\|_{L^2}$ but with the price of a logarithmic singularity $|\log\varepsilon|^{1+}$, see Lemma \ref{lmw2}. 
		 Such singularity will cause a growth of $|\log\vep|^{\f32+}$ in the linear estimate \eqref{prop2-1}, and that is why we need the logarithmic coefficients in the main result. The above process is also applied to treat commutators in the weighted estimate of the vorticity. We refer to Lemma \ref{prop2.7} for details.	
		\end{itemize}
The rest of paper is organized as follows. In Section 2 we will introduce the function $Z(y)$ and establish some related interpolation inequalities. In Section 3, we will show the linear stability which is the key step of the proof. The nonlinear stability and the proof of Theorem \ref{thm1} will be given in Section 4. 

{\bf Notations.}
Throughout this paper, the positive constants which are independent of $\vep$ are denoted by $C$ and $c$. It may vary from line to line.  The constants $C_a,C_b,\cdots$ represent the generic positive constants depending on $a,~b,\cdots$, respectively. We say $A\sim B$ if there exist two positive constants $C_1$ and $C_2$, such that $C_1A\leq B\leq C_2A$, and $A\sim_\eta B$ if the constants $C_1$ and $C_2$ depend on $\eta$. $A\lesssim B$ means that there exists a positive constant $C$ such that $A\leq CB$, and $A\lesssim_\eta B$ means that the constant $C$ depends on $\eta$.  For any complex number $a$, we denote by $\b{a}$ its complex conjugate. For any two complex value functions $f$ and $g$ which depend on $y$, the notation $\langle\text{ },\text{ } \rangle$ represents the standard $L^2(\mathbb{R}_+)$ inner product, i.e., $\langle f,g \rangle=\int_{0}^\infty f\b{g}\dd y.$ Finally, we denote $\|\cdot\|_{L^p}$ as the standard $L^p(\mathbb{R}_+)$-norm and $\|\cdot\|_{L^p(\Omega)}$ as $L^p(\Omega)$-norm.

\section{Weight function }

In this section, we specify the weight function $Z^{\frac{1}{2}}(y)$ through the construction of $Z(y)$ and establish some related interpolation inequalities. 
Recall $G_s(Y)=H_s^2(Y)-U_s^2(Y)$ and it is easy to obtain from the assumptions \eqref{A3}-\eqref{A4} that, 
\begin{align*}
\gamma_0\leq G_s(Y)\leq \bar{\gamma}^2,\qquad \sup_{Y\geq 0}(1+Y)^3|G'_s(Y)|\lesssim\bar{M}.
\end{align*}
We construct a $C^1$-function $\widetilde{G}(y),~y\in\R_+$ satisfying
\begin{align}\label{def-tG}
\widetilde{G}(y):=\left\{\begin{aligned}
\frac{1}{G_s(y/\sqrt{\varepsilon})}, \quad &0\leq y\leq 1,\\
0,\quad &y\geq 2, 
\end{aligned}\right.
\end{align}
and
\begin{align}\label{def-tG1}
\frac{1}{2\bar{\gamma}^2}\leq\widetilde{G}(y)\leq\frac{2}{\gamma_0},\quad\left|\widetilde{G}'(y)\right|\lesssim \bar{M}\varepsilon,~\mbox{for}~y\in\left[1,\frac{3}{2}\right];\qquad \widetilde{G}'(y)\leq 0,~\mbox{for}~y\in\left[\frac{3}{2},2\right].
\end{align}
It is not difficult to know such function $\widetilde{G}(y)$ exists due to the fact
\begin{align*}
	\left.\left(\frac{1}{G_s(y/\sqrt{\varepsilon})}\right)'\right|_{y=1}=\left.\varepsilon\cdot\left(-\frac{Y^3G'_s(Y)}{G_s^2(Y)}\right)\right|_{Y=\frac{1}{\sqrt{\varepsilon}}}\lesssim\bar{M}\varepsilon.
\end{align*}
Then we define the function 
\begin{align}\label{z}
Z(y):=\int_0^y\widetilde{G}(y')dy'.
\end{align}
One can see that $Z\in C^2(\mathbb{R}_+)$.
In the following lemma, we give some basic properties of $Z(y)$, which will be frequently used later.
\begin{lemma}\label{lmz}
There exists a positive constant $C_0$ independent of $\vep$ and $\b{M}$, such that the following estimates hold for $Z(y)$:\\
\newline
 (1) 
 $0\leq Z(y)\leq C_0,~ Z'(y)\geq0$ and
	\begin{align}\label{z1}
	C_0^{-1}y\leq Z(y)\leq C_0 y\quad\mbox{for}~y\in[0,2],\qquad Z(y)\equiv\int_{0}^{2}\widetilde{G}(y')dy'\triangleq\bar{Z} \quad\mbox{for}~y\geq2. 
	\end{align}	
(2) 
\begin{align}\label{z2}
G_s(\frac{y}{\sqrt{\varepsilon}})Z'(y)\equiv 1\quad\mbox{for}~y\in[0,1],\qquad |y^kZ''(y)|\leq C_0\bar{M}\vep^{\frac{k-1}{2}}\quad \mbox{for}~y\in[0,\frac{3}{2}],~0\leq k\leq 3,
\end{align}
and
	\begin{align}\label{z4}
-\left(G_s(\frac{y}{\sqrt{\varepsilon}})Z'(y)\right)'\geq -C_0\b{M}\vep \quad \mbox{for}~y\geq 1;\qquad Z''(y)\leq 0,\quad \mbox{for}~y\geq\frac{3}{2}.
	\end{align}
(3) 
	\begin{align}\label{z5}
	|(1+y)Z'(y)|\leq C_0, \quad |yZ''(y)|\leq C_0(\b{M}+1) \quad \mbox{ for }y\in\mathbb{R}_+.
	\end{align}
\end{lemma}
\begin{proof}
	It suffices to show \eqref{z4} and \eqref{z5}, since the other estimates are straightforward. As $Z'(y)=\widetilde{G}(y), Z''(y)=\widetilde{G}'(y)$, from  \eqref{def-tG1} for $\widetilde{G}$ it is easy to get $Z''(y)\leq0$ for $y\geq\frac{3}{2}.$ Then, we have that for $y\geq1$,
	$$-\left(G_s(\frac{y}{\sqrt{\varepsilon}})Z'(y)\right)'=-G_s(\frac{y}{\sqrt{\varepsilon}})\widetilde{G}'(y)-\varepsilon^{-\frac{1}{2}}G'_s(\frac{y}{\sqrt{\varepsilon}})\widetilde{G}(y).
	$$
	By using \eqref{def-tG1}, it implies  
	\begin{align*}
			\left|G_s(\frac{y}{\sqrt{\varepsilon}})\widetilde{G}'(y)\right|\leq C_0\varepsilon,\quad \mbox{for}~1\leq y\leq\frac{3}{2};\qquad G_s(\frac{y}{\sqrt{\varepsilon}})\widetilde{G}'(y)\leq 0,\quad \mbox{for}~y\geq\frac{3}{2},
	\end{align*}
	and
	$$
	\left|G'_s(\frac{y}{\sqrt{\varepsilon}})\widetilde{G}(y)\right|\leq C_0\varepsilon^{\frac{3}{2}}\frac{|Y^3G'_s(Y)|}{y^3}\leq C_0\bar{M}\varepsilon^{\frac{3}{2}} y^{-3}
	$$
	with $Y=y/\sqrt{\varepsilon}$. Therefore, \eqref{z4} follows immediately. The proof of \eqref{z5} is similar and we omit it for brevity. This completes the proof of the lemma. 
\end{proof}
Next we establish an interpolation inequality which is analogous to Proposition 2.4 in \cite{GM}. 
\begin{lemma}\label{lmz2}
	Let $Z(y)$ be the weight function defined in \eqref{z} and $C_0$ be the positive constant given in Lemma \ref{lmz}. 
	It holds that for any $g\in H^{1}(\mathbb{R}_+)$,
	\begin{align}\label{AP1}
	\|g\|_{L^2}\leq 2\sqrt{2C_0}\|Z^{\f12}g\|_{L^{2}}^{\f23}\|\pa_yg\|_{L^2}^{\f13}+C_0\|Z^{\f12}g\|_{L^2}.
	\end{align}
\end{lemma}
\begin{proof} 
	Since $Z(y)\sim y$ for $y\in [0,2]$ and $Z(y)\equiv \bar{Z}$ for $y\geq 2$, the inequality \eqref{AP1} follows from a similar argument as that in \cite{GHM,GM}. Nevertheless, we give its proof for completeness. Let $0<\eta\leq 2$ be a constant which will be chosen later. Then from \eqref{z1} one has
	\begin{align}\label{inter}\begin{aligned}
	\|g\|_{L^2}^2&=\int_0^\eta |g(y)|^2d y+\int_\eta^\infty |g(y)|^2d y\leq \eta\|g\|_{L^\infty}^2+\int_\eta^\infty\frac{1}{Z(y)} Z(y)|g(y)|^2d y\\
	&\leq 2\eta\|g\|_{L^2}\|\pa_yg\|_{L^2}+C_0\eta^{-1}\|Z^{\f12}g\|_{L^2}^2, 
	\end{aligned}\end{align}
	where we have used the fact 
	$$Z(y)\geq Z(\eta)\geq C_0^{-1}\eta \quad \mbox{for}~y\geq\eta,$$
	and the classical interpolation inequality $\|g\|_{L^\infty}^2\leq 2\|g\|_{L^2}\|\pa_yg\|_{L^2}.$
	Then we optimize the right-hand side of \eqref{inter} with respect to $\eta\in(0,2]$. On one hand, when $\displaystyle\frac{\sqrt{C_0}\|Z^{\f12}g\|_{L^2}}{\sqrt{2\|g\|_{L^2}\|\pa_yg\|_{L^2}}}\leq2,$ we choose $\displaystyle\eta=\frac{\sqrt{C_0}\|Z^{\f12}g\|_{L^2}}{\sqrt{2\|g\|_{L^2}\|\pa_yg\|_{L^2}}}$ in \eqref{inter} to obtain
	\begin{align*}
		\|g\|_{L^2}^2\leq 2\sqrt{2C_0} \|Z^{\f12}g\|_{L^2}\sqrt{\|g\|_{L^2}\|\pa_yg\|_{L^2}},
	\end{align*}
	which implies
	\begin{align}\label{inter1}
\|g\|_{L^2}\leq 2\sqrt{2C_0} \|Z^{\f12}g\|_{L^2}^{\frac{2}{3}}\|\pa_yg\|_{L^2}^{\frac{1}{3}}.
\end{align}
On the other hand,  when $\displaystyle\frac{\sqrt{C_0}\|Z^{\f12}g\|_{L^2}}{\sqrt{2\|g\|_{L^2}\|\pa_yg\|_{L^2}}}>2,$ it implies $\displaystyle\|g\|_{L^2}\|\pa_yg\|_{L^2}<\frac{C_0}{8}\|Z^{\f12}g\|_{L^2}^2.$ We apply this inequality to \eqref{inter} and let $\eta=2$ to get
\begin{align}\label{inter2}
\|g\|_{L^2}^2<C_0 \|Z^{\f12}g\|_{L^2}^{2}.
\end{align}
Combining \eqref{inter1} with \eqref{inter2} yields the desired estimate \eqref{AP1}.
\end{proof}

The following two lemmas are crucial for controlling some lower-order terms involving the commutator $[Z,\pa_y]$.

\begin{lemma}\label{lmw1}
	Let $R(y)$ be a $L^2$-function supported on $[0,2]$. Then for any $\eta>0$, there exists a positive constant $C_\eta$ such that
	\begin{align}\label{w1}
	\left|\langle R,\pa_y^{-1}h\rangle\right|
	\leq C_\eta\|Z^{\f12}R\|_{L^2}\|Z^{\f12}|\log Z|^{1+\f{\eta}{3}}h\|_{L^2}.
	\end{align}
\end{lemma}
\begin{proof}
By Cauchy-Schwarz inequality,
\begin{align}\label{w1-1}
\left|\langle R,\pa_y^{-1}h\rangle\right|
&\leq \|Z^{\f12}R\|_{L^2(0,2)}\|Z^{-\f12}\pa_y^{-1}h\|_{L^2(0,2)}.
\end{align} 
Recall the weighted Hardy inequality (\cite{KMP}):
\begin{align}\label{w-Hardy}
\left\|u^{\f1p}\partial_y^{-1}h\right\|_{L^p}\lesssim \left\|v^{\f1q}h\right\|_{L^q},\quad 1\leq p\leq q<\infty
\end{align}
provided that the weight functions $u(y)$ and $v(y)$ satisfy 
\begin{align*}
\sup_{y}\left(\int_{\{z\geq y\}}u(z)dz\right)^{\frac{1}{q}}\left(\int_{\{z\leq y\}}v(z)^{1-p'}dz\right)^{\frac{1}{p'}}<\infty,\quad \frac{1}{p}+\frac{1}{p'}=1.
\end{align*}
Then, it follows that by letting $p=q=2,~u(y)=Z^{-1}(y), ~v(y)=Z(y)\big|\log Z(y)\big|^{2+}$ in \eqref{w-Hardy}, 
\begin{align}\label{w1-2}
\|Z^{-\f12}\pa_y^{-1}h\|_{L^2(0,2)}\lesssim \|Z^{\f12}|\log Z|^{1+\f{\eta}{3}}h\|_{L^2(0,2)},
\end{align} 
which, along with \eqref{w1-1}, implies \eqref{w1} immediately.

For easy reference, we give an intuitive proof of \eqref{w1-2} instead of using weighted Hardy inequality \eqref{w-Hardy}. Note that \eqref{w1-2} holds automatically when $y$ is away from zero, since $Z(y)$ is bounded from below by a positive constant. Hence we only need to focus on the case of $y$ near zero. To this end,	
as $Z(y)\sim y$ with $y\in (0,1/2)$, for any $\eta>0$ we have the following pointwise estimate 
$$
\begin{aligned}
|\pa_y^{-1}h(y)|&\leq C\int_0^y \xi^{-\f12}|\log\xi|^{-(1+\f\eta3)}Z^{\f12}|\log Z|^{1+\f\eta3}|h(\xi)|d\xi\\
&\leq C\|Z^{\f12}|\log Z|^{1+\f{\eta}{3}}h\|_{L^2(0,2)}\cdot\left(\int_0^y\xi^{-1}|\log\xi|^{-2-\f{2\eta}{3}}d\xi\right)^{\f12}\\
&\leq C_\eta|\log y|^{-\frac{1}{2}-\f\eta3}\cdot\|Z^{\f12}|\log Z|^{1+\f{\eta}{3}}h\|_{L^2(0,2)}.
\end{aligned}
$$
Then, 
$$
\begin{aligned}
\left\|Z^{-\f12}\pa_y^{-1}h\right\|_{L^2(0,\f12)}
&\leq C~\left(\int_0^{\f12}y^{-1}|\pa_y^{-1}h(y)|^2dy\right)^{\f12}\\
&\leq C_\eta\|Z^{\f12}|\log Z|^{1+\f{\eta}{3}}h\|_{L^2(0,2)}\cdot\left(\int_0^2y^{-1}|\log y|^{-1-\frac{2\eta}{3}}dy\right)^{\f12}\\
&\leq C_\eta \|Z^{\f12}|\log Z|^{1+\f{\eta}{3}}h\|_{L^2(0,2)}.
\end{aligned}
$$
Thus we obtain  \eqref{w1-2}.
\end{proof}

\begin{lemma}\label{lmw2}
	For any $\eta>0$ and $\delta\geq0$, there exists a positive constant $C_{\eta,\delta}$ independent of $\vep$, such that
	\begin{align}\label{w2}
		\|Z^{\f12}|\log Z|^{1+\f{\eta}{3}}h\|_{L^2}\leq C_{\eta,\delta}|\log\vep|^{1+\frac{\eta}{3}}\left(\|Z^{\f12}h\|_{L^2}+\vep^{\f14+\delta}\|h\|_{L^2}\right).
	\end{align}
\end{lemma}
\begin{proof}
We divide the integration  interval into $[0,\vep^{\f12+2\delta}]$ and $[\vep^{\f12+2\delta},\infty)$. In the interval $[0,\vep^{\f12+2\delta}]$, it holds that $Z(y)\sim y$. Let $\xi(y):=y|\log y|^{2+\f{2\eta}{3}}$, then
$$\xi'(y)=|\log y|^{1+\f{2\eta}{3}}\left[|\log y|-2(1+\f\eta3)\right]>0.
$$
Consequently, it holds $|\xi(y)|\leq C\vep^{\f12+2\delta}|\log \vep|^{2+\f{2\eta}{3}}$, which implies that
$$\int_0^{\vep^{\f12+2\delta}}Z|\log Z|^{2+\f{2\eta}{3}}|h|^2d y\leq C\int_0^{\vep^{\f12+2\delta}}|\xi||h|^2d y\leq C\vep^{\f12+2\delta}|\log \vep|^{2+\f{2\eta}{3}}\|h\|_{L^2}^2.
$$
In the interval $[\vep^{\f12+2\delta},\infty)$, since $Z(y)$ is bounded, it yields $|\log Z|\leq C|\log \vep|$. Then it holds that
$$\int_{\vep^{\f12+2\delta}}^\infty Z|\log Z|^{2+\f{2\eta}{3}}|h|^2d y\leq C|\log \vep|^{2+\f{2\eta}{3}}\|Z^{\f12}h\|_{L^2}^2.
$$
By combining these two inequalities, we obtain \eqref{w2} and the proof of the lemma is completed.
\end{proof}
Combining \eqref{w1} with \eqref{w2} yields that for any $\eta>0,\delta\geq0$ and $L^2$ function $R$ supported on $[0,2]$,
\begin{align}\label{w1_1}
		\left|\langle R,\pa_y^{-1}h\rangle\right|
	\leq C_{\eta,\delta} |\log\vep|^{1+\frac{\eta}{3}}\|Z^{\f12}R\|_{L^2}\left(\|Z^{\f12}h\|_{L^2}+\vep^{\f14+\delta}\|h\|_{L^2}\right).
\end{align}

We now conclude this section with the following lemma about the equivalence between the weighted estimates on the full gradient of divergence-free vector field and the weighted estimates of its vorticity.
\begin{lemma}\label{lmw3}
	Let $\Fq=(q_1,q_2)$ be a divergence-free vector field in $\Omega$ satisfying $q_2|_{y=0}=0$. There exists a positive constant $C>0$ such that
	\begin{align}\label{w3}
	\|Z^{\f12}\nabla\Fq\|_{L^2(\Omega)}\leq C\|Z^{\f12}\omega_{\Fq}\|_{L^2(\Omega)},
	\end{align}
	where $\omega_{\Fq}=\pa_yq_1-\pa_xq_2$ is the vorticity of $\Fq$.
\end{lemma}
\begin{proof}
Since $\pa_yq_1=\omega_\Fq+\pa_xq_2$ and $\pa_yq_2=-\pa_xq_1$,  it suffices to prove
\begin{align}\label{w3-1}
	\|Z^{\f12}\pa_x\Fq\|_{L^2(\Omega)}\lesssim\|Z^{\f12}\omega_{\Fq}\|_{L^2(\Omega)}.
\end{align} 
 Let $\phi_{\Fq}$ be the stream function of $\Fq$ and $\phi_{\Fq}$ is determined by 
\begin{align}\label{wq1}
\Delta\phi_{\Fq}=\omega_{\Fq}, \ \ \mbox{in}~ \Omega~;\qquad \phi_{\Fq}|_{y=0}=0.
\end{align}
For convenience, we introduce $\tilde{Z}(y):=\frac{y}{1+y}$. According to \eqref{z1}, we can find two positive constants $\underline{c}$ and $\b{c}$, such that
\begin{align*}
\underline{c}Z(y)\leq \tilde{Z}(y)\leq \b{c}Z(y),
\end{align*}
which implies the equivalence between the norms $\|Z^{\f12}f\|_{L^2(\Omega)}$ and $\|\tilde{Z}^{\f12}f\|_{L^2(\Omega)}$. 
Thus, we only need to show \eqref{w3-1} for the weight $\tilde{Z}^{\f12}$.  
By taking inner product of \eqref{wq1} with $\tilde{Z}\pa_{x}^2\phi_{\Fq}$, one has
$$\int_{\Omega}\tilde{Z}\pa_{x}^2\phi_{\Fq}\Delta\phi_{\Fq}d xd y=\int_{\Omega}\tilde{Z}\pa_{x}^2\phi_{\Fq}\omega_{\Fq}d xd y.
$$
It follows from Cauchy-Schwarz inequality that
$$\left|\int_{\Omega}\tilde{Z}\pa_{x}^2\phi_{\Fq}\omega_{\Fq}d xd y\right|\leq \|\tilde{Z}^{\f12}\pa_{x}^2\phi_{\Fq}\|_{L^2(\Omega)}\|\tilde{Z}^{\f12}\omega_{\Fq}\|_{L^2(\Omega)}\leq  \|\tilde{Z}^{\f12}\pa_{x}q_2\|_{L^2(\Omega)}\|\tilde{Z}^{\f12}\omega_{\Fq}\|_{L^2(\Omega)}.
$$
By integration by parts and using the fact that $\pa_y^2\tilde{Z}(y)=-\frac{2}{(1+y)^3}\leq 0$, it yields
$$
\begin{aligned}
\int_{\Omega}\tilde{Z}\pa_{x}^2\phi_{\Fq}\Delta\phi_{\Fq}d xd y&=\int_{\Omega}\tilde{Z}|\pa_{x}^2\phi_{\Fq}|^2d xd y-\int_{\Omega}\tilde{Z}\pa_x\phi_\Fq\pa_y^2(\pa_x\phi_\Fq) d xd y\\
&=\int_{\Omega}\tilde{Z}\left[|\pa_{x}^2\phi_{\Fq}|^2+|\pa_{xy}^2\phi_{\Fq}|^2\right]d xd y+\int_{\Omega}\pa_y\tilde{Z}\pa_{xy}^2\phi_\Fq\pa_x\phi_\Fq d xd y\\
&=\int_{\Omega}\tilde{Z}\left|\pa_{x}\nabla\phi_\Fq\right|^2d xd y-\f12\int_{\Omega}\pa_y^2\tilde{Z}|\pa_{x}\phi_\Fq|^2d xd y\\
&\geq \left\|\tilde{Z}^{\f12}\pa_{x}\Fq\right\|_{L^2(\Omega)}^2.
\end{aligned}
$$
Combining the above two inequalities implies that $\|\tilde{Z}^{\f12}\pa_x\Fq\|_{L^2(\Omega)}\leq C\|\tilde{Z}^{\f12}\omega_{\Fq}\|_{L^2(\Omega)}$, and \eqref{w3-1} follows. Therefore, we obtain \eqref{w3} and complete the proof of the lemma. 
\end{proof}

\section{Linear stability}

To obtain the solution to nonlinear problem \eqref{1.2}, we first consider the following linearized system 
\begin{equation}
\left\{
\begin{aligned}\label{2.0}
&U_s\pa_x{\FU}+{v}\pa_yU_s\Fe_1-H_s\pa_x{\FH}-{g}\pa_yH_s\Fe_1+ \nabla P-\mu\vep\Delta{\FU}=\Ff,\\
&U_s\pa_x{\FH}+{v}\pa_yH_s\Fe_1-H_s\pa_x{\FU}-{g}\pa_yU_s\Fe_1-\ka\vep\Delta{\FH}=\Fq,\\
&\nabla\cdot {\FU}=\nabla\cdot {\FH}=0,\\
&\FU|_{y=0}=(\partial_y h,g)|_{y=0}=\mathbf{0},
\end{aligned}
\right.
\end{equation}
where $\Ff=(f_1,f_2)$ and $\Fq=(q_1,q_2)$ are given inhomogeneous source terms. 
 Since $U_s$ and $H_s$ are independent of $x$, it is convenient to take Fourier transform in $x$ for \eqref{2.0} and study the following equivalent system:
\begin{equation}
\left\{
\begin{aligned}\label{2.1}
&i\tn U_s\FU_n+v_n\pa_yU_s\Fe_1-i\tn H_s\FH_n-g_n\pa_yH_s\Fe_1+(i\tn P_n,\pa_yP_n)-\mu\vep(\pa_y^2-\tn^2)\FU_n=\Ff_n,\\
&i\tn U_s\FH_n+v_n\pa_yH_s\Fe_1-i\tn H_s\FU_n-g_n\pa_yU_s\Fe_1-\ka\vep(\pa_y^2-\tn^2)\FH_n=\Fq_n,\\
&i\tn u_n+\pa_y v_n=i\tn h_n+\pa_yg_n=0,\\
&(u_n,v_n,\partial_yh_n,g_n)|_{y=0}=\mathbf{0}.
\end{aligned}
\right.
\end{equation}
Here $n\in\Z,~\tilde{n}=\f{n}{\varrho}, ~\FU_n=\FU_n(y)=\big(u_n(y),v_n(y)\big)$ and $\FH_n=\FH_n(y)=\big(h_n(y),g_n(y)\big)$ are $n$-th Fourier coefficients of the  velocity field $\FU(x,y)$ and magnetic field $\FH(x,y)$ respectively; and $\Ff_n=\Ff_n(y)=\big(f_{1,n}(y),f_{2,n}(y)\big)$, $\Fq_n=\Fq(y)=\big(q_{1,n}(y),q_{2,n}(y)\big)$ correspond to  $\Ff(x,y)$ and $\Fq(x,y)$ respectively. 
Moreover, it is not difficult to check that the following compatibility condition for $\Fq$ is needed: 
\begin{align}\label{ass-q}	\nabla\cdot \Fq=0,~ q_2|_{y=0}=0,\end{align}
and then $q_{2,0}\equiv0$ as a direct consequence of \eqref{ass-q}. 

For  simplicity of notation, we set $\FW=(\FU,\FH)$ and $\FW_n=(\FU_n,\FH_n)$. 
Let $\CI$ and $\pa_y^{-1}$ be anti-derivative operators defined by 
$$\CI f(y)=-\int_y^\infty f(y')d y',\qquad \pa_y^{-1}f(y)=\int_0^y f(y')d y'$$  respectively for any $f\in L^1(\mathbb{R}_+)$. Recall the solution space $\CX$ and its norm defined in \eqref{X} and \eqref{1.1-3} respectively. The solvability of the linear problem \eqref{2.0} is given by the following proposition.
\begin{proposition}\label{prop2.2}
		There exist positive constants  $\delta_1$ and $\vep_1$, such that the following statement holds. If
		$$
		\varrho(\b{M}+\b{M}^4)\leq \delta_1,\qquad \vep\in (0,\vep_1),
		$$
then for any $(\Ff,\Fq)$ 
satisfying \eqref{ass-q} and \begin{align}\label{2.3}
		(\CI f_{1,0},~\pa_y^{-1}q_{1,0})\in L^1(\mathbb{R}_+)\cap L^2(\mathbb{R}_+), \quad\CQ_{0}(\Ff, \Fq)\in L^2(\Omega),
		\end{align}
		the linear problem \eqref{2.0} admits a unique solution $\FW\in \CX$  that satisfies for any $\eta>0$:
		\begin{align}\label{prop2-1}
		\begin{aligned}
		\|\FW\|_{\CX}\leq& C\vep^{-1}\left[\left\|\big(\CI f_{1,0},\pa_y^{-1}q_{1,0}\big)\right\|_{L^1}+\vep^{\f14}\left\|\big(\CI f_{1,0},\pa_y^{-1}q_{1,0}\big)\right\|_{L^2}+\big\|Z^{\f12}\big(\CI f_{1,0},\pa_y^{-1}q_{1,0}\big)\big\|_{L^2}\right]\\
		&+C\vep^{-\f14}|\log\vep|^{\frac{3+\eta}{2}}\left[\|\CQ_0(\Ff,\Fq)\|_{L^2(\Omega)}+\vep^{-\f14}\|Z^{\f12}\CQ_0(\Ff,\Fq)\|_{L^2(\Omega)}\right].
		\end{aligned}\end{align}
Here the positive constant $C$ is independent of $\vep$.
\end{proposition}
The following three subsections are devoted to the proof of Proposition \ref{prop2.2}.
\subsection{Estimate on zero mode}
We first consider the zero-mode $(\FU_0,\FH_0)$. When $n=0$, the system \eqref{2.1} reduces to the following simple ODE system:
\begin{equation}
\left\{
\begin{aligned}\label{2.2}
&v_0\pa_yU_s-g_0\pa_yH_s-\mu\vep\pa_y^2u_0=f_{1,0},\\
&\pa_yp_0-\mu\vep\pa_y^2v_0=f_{2,0},\\
&v_0\pa_yH_s-g_0\pa_yU_s-\ka\vep\pa_y^2h_0=q_{1,0},\\
&-\ka\vep\pa_y^2g_0=0,\\
&\pa_yv_0=\pa_yg_0=0,\\
&(u_0,v_0,\partial_yh_0,g_0)|_{y=0}=\mathbf{0}.
\end{aligned}
\right.
\end{equation}
We can explicitly solve \eqref{2.2} to have $v_0=g_0=0$ and
\begin{align*}
\begin{aligned}
u_0=\frac{1}{\mu\vep}\int_0^y\int_{y'}^{+\infty}f_{1,0}(y'')d y''d y'=-\frac{1}{\mu\vep}\int_0^y\CI f_{1,0}(y')d y',\\ h_0=\frac{1}{\ka\vep}\int_y^\infty\int_0^{y'}q_{1,0}(y'')d y''d y'=\frac{1}{\ka\vep}\int_y^\infty\pa_y^{-1}q_{1,0}(y')d y'.
\end{aligned}\end{align*}
As a direct consequence, one has the lemma.
\begin{lemma}\label{prop2.1}
	From \eqref{2.3} it holds that
	\begin{align*}
	\|(u_0,h_0)\|_{L^\infty}&\leq C\vep^{-1}\big\|\big(\CI f_{1,0},\pa_y^{-1}q_{1,0}\big)\big\|_{L^1},\\
	 \|(\pa_yu_0,\pa_yh_0)\|_{L^2}&\leq C\vep^{-1}\big\|\big(\CI f_{1,0},\pa_y^{-1}q_{1,0}\big)\big\|_{L^2},\\
	 \big\|Z^{\f12}(\pa_yu_0,\pa_yh_0)\big\|_{L^2}&\leq C\vep^{-1}\big\|Z^{\f12}\big(\CI f_{1,0},\pa_y^{-1}q_{1,0}\big)\big\|_{L^2}.
	\end{align*}
\end{lemma}

\subsection{Estimate on non-zero mode}
Next we consider non-zero mode $(\FU_n, \FH_n),~n\neq0$. Since $\FH=(h,g)$ is divergence-free, there exists a stream function $\psi(x,y)$, such that 
$$h=\pa_y\psi, \quad g=-\pa_x\psi,\quad \psi|_{y=0}=0,
$$ 
and the equation of $\psi$ is given by
\begin{align*}
U_s\pa_x\psi+H_sv-\ka\vep\Delta\psi=\pa_y^{-1}q_1.
\end{align*}
Inspired by \cite{LXY2}, we denote by
 $$\displaystyle a_p(Y)=\frac{U_s(Y)}{H_s(Y)},\qquad b_p(Y)=\frac{\partial_YH_s(Y)}{H_s(Y)},$$
 and introduce 
new ``good unknown function" $\widehat{\FW}=(\widehat{\FU},\widehat{\FH})=(\hat{u},\hat{v},\hat{h},\hat{g})$:
\begin{align}\label{new-unknown}
\begin{cases}
\hat{u}(x,y)&=u(x,y)-\pa_y\big(a_p(Y)\psi(x,y)\big),\\ \hat{v}(x,y)&=v(x,y)+\pa_x\big(a_p(Y)\psi(x,y)\big),\\
\hat{h}(x,y)&=\pa_y\left(\f{\psi(x,y)}{H_s(Y)}\right)=\frac{1}{H_s(Y)}\left(h(x,y)-\varepsilon^{-\frac{1}{2}}b_p(Y) \psi(x,y)\right),\\
\displaystyle \hat{g}(x,y)&=-\pa_x\left(\f{\psi(x,y)}{H_s(Y)}\right)=\frac{g(x,y)}{H_s(Y)}
\end{cases}
\end{align}
with $Y=\frac{y}{\sqrt{\varepsilon}}.$ Also, denote by 
\begin{align*}
	\hat{\psi}(x,y)=\frac{\psi(x,y)}{H_s(Y)},
\end{align*}
and it is easy to check   that $\hat{\psi}$ is the stream function of $\widehat{\FH}$.
Then by this transformation and some tedious calculations we can rewrite  \eqref{2.0} into the following problem for $\widehat{\FW}$:
\begin{align}\label{pr-new}
	\begin{cases}
\displaystyle
	(1+\frac{\mu}{\kappa})U_s\pa_x\widehat{\FU}-
	G_s\pa_x\widehat{\FH}+\varepsilon^{\f12}\big(\bm{A}_{\FU}\pa_{x}\widehat{\FH}+\bm{B}_{\FU}\pa_{y}\widehat{\FH}\big)+\bm{C}_{\FU}\widehat{\FH}+\varepsilon^{-\f12}\hat{\psi}\FD_{\FU}
	+\nabla P=\mu\vep\Delta\widehat{\FU}+\FR_{\FU},\\
\displaystyle	-\pa_x\widehat{\FU}-2\kappa\varepsilon^{\f12}b_p\partial_y\widehat{\FH}+\bm{C}_{\FH}\widehat{\FH}+\varepsilon^{-\f12}\hat{\psi}\FD_{\FH}=\ka\vep\Delta\widehat{\FH}+\FR_{\FH},\\
\displaystyle	\nabla\cdot \widehat{\FU}=\nabla\cdot \widehat{\FH}=0,\\
\displaystyle	\widehat{\FU}|_{y=0}=(\partial_y \hat{h},\hat{g})|_{y=0}=\mathbf{0}.
	\end{cases}
\end{align}
Here $\bm{A}_{\FU}, \bm{B}_{\FU}, \bm{C}_{\FU}, \bm{C}_{\FH}$ are matrices and $\FD_{\FU}, \FD_{\FH}$ are vectors. These terms depend only on $\mu,\kappa,U_s, H_s$, and they have the following forms:
\begin{align}\label{form-metrics}
\begin{aligned}
&\bm{A}_\FU=\begin{pmatrix}
0&(\mu-\kappa)\pa_YU_s\\
0&0
\end{pmatrix},\quad \bm{B}_\FU=\begin{pmatrix}
(\kappa-3\mu)\pa_YU_s+2\mu a_p\pa_YH_s & 0\\
0&2\mu (a_p\pa_YH_s-\pa_YU_s)
\end{pmatrix},\\
&\bm{C}_\FU=\frac{1}{H_s}\begin{pmatrix}
2\kappa \pa_YH_s\pa_YU_s-2\mu a_p(\pa_YH_s)^2+3\mu(U_s\pa_Y^2H_s-H_s\pa_Y^2U_s) & 0\\
0&\mu(U_s\pa_Y^2H_s-H_s\pa_Y^2U_s)
\end{pmatrix},\\
&\bm{C}_\FH=\frac{\kappa}{H_s^2}\begin{pmatrix}
2(\pa_YH_s)^2-3H_s\pa_Y^2H_s&0\\
0&-H_s\pa_Y^2H_s
\end{pmatrix},
\end{aligned}
\end{align}
and
\begin{align}\label{form-vec}
\begin{aligned}
	&\FD_\FU=\frac{1}{H_s}\left(\kappa\pa_YU_s\pa_Y^2H_s-\mu a_p\pa_YH_s\pa_Y^2H_s+\mu U_s\pa_Y^3H_s-\mu H_s\pa_Y^3U_s ,~0\right)^T,\\
	&\FD_\FH=\frac{\kappa}{H_s^2}\left(\pa_YH_s\pa_Y^2H_s-H_s\pa_Y^3H_s,~0\right)^T.
\end{aligned}\end{align}
The source term $\FR\triangleq(\FR_\FU,\FR_\FH)=(R_u,R_v,R_h,R_g)$ is given by
\begin{align*}
\begin{aligned}
	&\FR_\FU=(R_u,R_v)=\left(f_1-\frac{\mu}{\kappa}a_pq_1+\frac{\varepsilon^{-\f12}}{H_s}\left(\frac{\mu}{\kappa}a_p\pa_YH_s-\pa_YU_s\right)\pa_y^{-1}q_1,~f_2-\frac{\mu}{\kappa}a_pq_2\right)^T,\\
	&\FR_\FH=(R_h,R_g)=\frac{1}{H_s}\left(q_1-\varepsilon^{-\f12}b_p\pa_y^{-1}q_1,~q_2\right)^T,
\end{aligned}\end{align*}
where the divergence-free condition $\nabla\cdot\Fq=0$ has been used. 

Now, let us turn to the Fourier mode. According to \eqref{new-unknown}, the $n$-th Fourier coefficients $\widehat{\FW}_n=(\widehat{\FU}_n,\widehat{\FH}_n)\triangleq(\hu,\hv,\hh,\hg)$ of $\widehat{\FW}$ are given by
\begin{align}\label{n-unknown}
\begin{cases}
\hu(y)=u_n(y)-\pa_y\big(a_p(Y)\psi_n(y)\big), \\
\hv(y)=v_n(y)+i\tn a_p(Y)\psi_n(y),\\
\hh(y)=\pa_y\hat{\psi}_n(y)
=\frac{1}{H_s(Y)}\left(h_n(y)-\varepsilon^{-\frac{1}{2}}b_p(Y) \psi_n(y)\right),\\ 
\hg(y)=-i\tn\hat{\psi}_n(y)
=\frac{g_n(y)}{H_s(Y)}.
\end{cases}\end{align}
Here $\hat{\psi}_n(y)$ and $\psi_n(y)$ are the $n$-th Fourier coefficients of $\hat{\psi}(x,y)$ and $\psi(x,y)$ respectively, and it holds that  $\displaystyle \hat{\psi}_n(y)=\frac{\psi_n(y)}{H_s(Y)}$. Then, we obtain by taking the Fourier transformation in the problem \eqref{pr-new} that
\begin{align}\label{pr-new-n}
\begin{cases}
\displaystyle
i\tn\left[(1+\frac{\mu}{\kappa})U_s\widehat{\FU}_n-
G_s\widehat{\FH}_n+\varepsilon^{\f12}\bm{A}_{\FU}\widehat{\FH}_n\right]+\varepsilon^{\f12}\bm{B}_{\FU}\pa_{y}\widehat{\FH}_n+\bm{C}_{\FU}\widehat{\FH}_n+\varepsilon^{-\f12}\hat{\psi}_n\FD_{\FU}\\
\displaystyle
\hspace{4.65cm}+(i\tn p_n,\pa_yp_n)^T-\mu\vep(\pa_y^2-\tn^2)\widehat{\FU}_n=\FR_{\FU,n},\\
\displaystyle	-i\tn\widehat{\FU}_n-2\kappa\varepsilon^{\f12}b_p\partial_y\widehat{\FH}_n+\bm{C}_{\FH}\widehat{\FH}_n+\varepsilon^{-\f12}\hat{\psi}_n\FD_{\FH}-\ka\vep(\pa_y^2-\tn^2)\widehat{\FH}_n=\FR_{\FH,n},\\
\displaystyle	i\tn\hat{u}_n+\pa_y\hat{v}_n=i\tn\hat{h}_n+\pa_y\hat{g}_n=0,\\
\displaystyle	\widehat{\FU}_n|_{y=0}=(\partial_y \hat{h}_n,\hat{g}_n)|_{y=0}=\mathbf{0},
\end{cases}
\end{align}
with the source $\FR_n\triangleq(\FR_{\FU,n}, \FR_{\FH,n})=(R_{u,n},R_{v,n},R_{h,n},R_{g,n})$: 
\begin{align}\label{form-source-n}
\begin{aligned}
&\FR_{\FU,n}=(R_{u,n},R_{v,n})=\left(f_{1,n}-\frac{\mu}{\kappa}a_pq_{1,n}+\frac{\varepsilon^{-\f12}}{H_s}\big(\frac{\mu}{\kappa}a_p\pa_YH_s-\pa_YU_s\big)\pa_y^{-1}q_{1,n},~f_{2,n}-\frac{\mu}{\kappa}a_pq_{2,n}\right)^T,\\
&\FR_{\FH,n}=(R_{h,n},R_{g,n})=\frac{1}{H_s}\left(q_{1,n}-\varepsilon^{-\f12}b_p\pa_y^{-1}q_{1,n},~q_{2,n}\right)^T.
\end{aligned}\end{align}
Before we  estimate $\widehat{\FW}_n$  in the new system \eqref{pr-new-n},  let us explain why $\widehat{\FW}$ defined by \eqref{new-unknown} is a  ``good unknown function". For this, we first show in next lemma the equivalence between the original unknown $\FW_n$ and the newly defined 
$\widehat{\FW}_n$. The proof is similar as \cite{LXY2}, we put it into the Appendix.
\begin{lemma}\label{lem2.3}
	For any $1<p\leq \infty$, it holds that
	\begin{align}
	\|\FW_n\|_{L^p}\sim_{\bar{M}} \|\widehat{\FW}_n\|_{L^p}.\label{lem2.3-1}
	\end{align}
	Moreover, we have
	\begin{align*}
	\|Z^{\f12}\FW_n\|_{L^2}+\vep^{\f14}\|\FW_n\|_{L^2}&\sim_{\bar{M}} \|Z^{\f12}\widehat{\FW}_n\|_{L^2}+\vep^{\f14}\|\widehat{\FW}_n\|_{L^2},\\
	\|\FW_n\|_{L^2}+\vep^{\f12}\|\pa_y\FW_n\|_{L^2}&\sim_{\bar{M}} \|\widehat{\FW}_n\|_{L^2}+\vep^{\f12}\|\pa_y\widehat{\FW}_n\|_{L^2},\\
	\|\FW_n\|_{L^2}+\vep^{\f14}\|Z^{\f12}(\pa_y\FW_n,i\tilde{n}\FW_n)\|_{L^2}&\sim_{\bar{M}} \|\widehat{\FW}_n\|_{L^2}+\vep^{\f14}\|Z^{\f12}(\pa_y\widehat{\FW}_n,i\tilde{n}\widehat{\FW}_n)\|_{L^2}.
	\end{align*}	
\end{lemma}
Next, the following  lemma states that the coefficient matrices and vectors in the system \eqref{pr-new-n} are of $O(1)$.
\begin{lemma}\label{lm2.1}
	There exists a positive constant $C$ independent of $\vep$, such that 
	\begin{align}\label{R}
	\begin{aligned}
	&\|(1+Y)\bm{A}_\FU\|_{L^\infty_Y}
	+\|(1+Y)\bm{B}_{\FU}\|_{L^\infty_Y}\leq C\b{M},\\
	&\|(1+Y)\bm{C}_\FU\|_{L^\infty_Y}+\|(1+Y)\bm{C}_\FH\|_{L^\infty_Y}+\|(1+Y)^2\FD_\FU\|_{L^\infty_Y}+\|(1+Y)^2\FD_\FH\|_{L^\infty_Y}\leq C\b{M}(1+\b{M}),\\ 
	&\|\FR_n\|_{L^2}+\vep^{-\f14}\|Z^{\f12}\FR_n\|_{L^2}\leq C(1+\b{M})\left(\|(\Ff_n,\Fq_n)\|_{L^2}+\vep^{-\f14}\|Z^{\f12}(\Ff_n,\Fq_n)\|_{L^2}\right).
	\end{aligned}
	\end{align}
\end{lemma}
\begin{proof} We sketch the proof by showing the
	 estimate on $\FR_n$ because other estimates follow directly from \eqref{A3} and the formulation \eqref{form-metrics}, \eqref{form-vec}. According to the expression in \eqref{form-source-n},  we treat the term $\varepsilon^{-\f12}b_p\pa_y^{-1}q_{1,n}$ as an example. First by \eqref{A3} and Hardy inequality, it holds that
	$$
	\begin{aligned}
	\|\varepsilon^{-\f12}b_p\pa_y^{-1}q_{1,n}\|_{L^2}&\leq \|Yb_p\|_{L^\infty_Y}\|y^{-1}\pa_y^{-1}q_{1,n}\|_{L^2}\leq\b{M}\|q_{1,n}\|_{L^2},
	\end{aligned}
	$$
	and by \eqref{z1}, 
	$$
	\begin{aligned}
	&\left\|Z^{\f12}\left(\varepsilon^{-\f12}b_p\pa_y^{-1}q_{1,n}\right)\right\|_{L^2}\leq\vep^{\f14}\left\|\sqrt{\frac{Z(y)}{y}}\right\|_{L^\infty}\|Y^{\f32}b_p\|_{L^\infty_Y}\|y^{-1}\pa_y^{-1}q_{1,n}\|_{L^2}\leq  	C\b{M}\vep^{\f14}\|q_{1,n}\|_{L^2}.
	\end{aligned}
	$$
Hence, we obtain the  estimate on $\FR_{n}$. 
\end{proof}

We are now ready  to establish the uniform-in-$\varepsilon$ estimate on  $\widehat{\FW}_n$ through \eqref{pr-new-n}. As the first step, the following lemma gives  the $L^2$-estimate on the full derivatives of $\widehat{\FW}_n$.

\begin{lemma}\label{prop2.3}
 Let $\widehat{\FW}_n$ be the $H^1$-solution of the linear problem \eqref{pr-new-n}. 
 There exists a positive constant
 $C_3$ independent of $\vep$, $\tn$ and $\b{M}$, such that
	\begin{align}\label{2.6}
	\begin{aligned}
	\sqrt{\vep}\left(\|\pa_y\widehat{\FW}_n\|_{L^2}+|\tn|\|\widehat{\FW}_n\|_{L^2}\right)\leq C_3\b{M}^{\f12}(1+\b{M}^{\f12})\|\widehat{\FW}_n\|_{L^2}+C_3\|\FR_n\|^{\f12}_{L^2}\|\widehat{\FW}_n\|_{L^2}^{\f12}.
	\end{aligned}\end{align}	
\end{lemma}
\begin{proof}
We take inner product of the first equality for $\widehat{\FU}_n$ in \eqref{pr-new-n} with $\widehat{\FU}_n$, and the second equality for $\widehat{\FH}_n$ in \eqref{pr-new-n} with $\displaystyle G_s(\frac{y}{\sqrt{\varepsilon}})\widehat{\FH}_n$ respectively, and then take the summation
of these two equations.  The real part of the final equation gives
\begin{align}\label{L2_y-0}
\begin{aligned}
&\text{Re}\left\langle i\tn\vep^{\f12} \bm{A}_\FU\widehat{\FH}_n+\vep^{\f12}\bm{B}_\FU\pa_y \widehat{\FH}_n+\bm{C}_\FU\widehat{\FH}_n+\vep^{-\f12}\hat{\psi}_n\FD_\FU,~{\widehat{\FU}}_n\right\rangle-\text{Re}\left\langle \mu\vep (\pa_y^2-\tn^2) \widehat{\FU}_n,~{\widehat{\FU}}_n\right\rangle\\
&+\text{Re}\left\langle -2\kappa\vep^{-\f12}b_p\pa_y\widehat{\FH}_n+\bm{C}_\FH\widehat{\FH}_n+\vep^{-\f12}\hat{\psi}_n\FD_\FH,~G_s{\widehat{\FH}}_n\right\rangle-\text{Re}\left\langle\kappa \vep(\pa_y^2-\tn^2)\widehat{\FH}_n ,~G_s{\widehat{\FH}}_n\right\rangle\\
&=\text{Re}\left\langle \FR_{\FU,n},~{\widehat{\FU}}_n\right\rangle+\text{Re}\left\langle \FR_{\FH,n},~G_s{\widehat{\FH}}_n\right\rangle.
\end{aligned}\end{align}
Here we have used the fact 
$$\left\langle (i\tn p_{n},\pa_yp_n)^{T},~ \widehat{\FU}_n\right\rangle=0,$$ 
which follows from the integration by parts, divergence-free condition $i\tn \hat{u}_n+\pa_y\hat{v}_n=0$ and the boundary condition $\hat{v}_n|_{y=0}=0$.

 Next we estimate terms in  \eqref{L2_y-0}. For the diffusion terms, by integration by parts and the boundary condition $\widehat{\FU}_n|_{y=0}=(\pa_y\hat{h}_n,\hat{g}_n)|_{y=0}=\mathbf{0}$, we write
 \begin{align*}
 	\begin{aligned}
 	-\text{Re}\left\langle \mu\vep (\pa_y^2-\tn^2) \widehat{\FU}_n,~{\widehat{\FU}}_n\right\rangle=\mu\vep\left(\tn^2\|\widehat{\FU}_n\|_{L^2}^2+\|\pa_y\widehat{\FU}_n\|_{L^2}^2\right).
 	\end{aligned}
 \end{align*}
 Then by using \eqref{A3} and \eqref{A4}, we have 
 \begin{align*}
 \begin{aligned}
 -\text{Re}\left\langle\kappa \vep(\pa_y^2-\tn^2)\widehat{\FH}_n ,~G_s{\widehat{\FH}}_n\right\rangle&\geq \gamma_0\kappa\vep\left(\tn^2\|\widehat{\FH}_n\|_{L^2}^2+\|\pa_y\widehat{\FH}_n\|_{L^2}^2\right)+\kappa \vep^{\f12
 }\text{Re}\left\langle\pa_y\widehat{\FH}_n ,~\pa_yG_s{\widehat{\FH}}_n\right\rangle\\
&\geq\frac{\gamma_0\kappa}{2}\vep\left(\tn^2\|\widehat{\FH}_n\|_{L^2}^2+\|\pa_y\widehat{\FH}_n\|_{L^2}^2\right)-C\bar{M}^2\|\widehat{\FH}_n\|_{L^2}^2.
 \end{aligned}
 \end{align*}
By Cauchy-Schwarz inequality and the bound given in Lemma \ref{lm2.1}, it holds that
\begin{align*}
\begin{aligned}
&\left|\left\langle i\tn\vep^{\f12} \bm{A}_\FU\widehat{\FH}_n+\vep^{\f12}\bm{B}_\FU\pa_y \widehat{\FH}_n+\bm{C}_\FU\widehat{\FH}_n,~{\widehat{\FU}}_n\right\rangle\right|\\
&\leq\left[\vep^{\f12}\left(|\tn|\|\bm{A}_\FU\|_{L^\infty_Y}\|\widehat{\FH}_n\|_{L^2}+\|\bm{B}_\FU\|_{L^\infty_Y}\|\pa_y\widehat{\FH}_n\|_{L^2}\right)+\|\bm{C}_\FU\|_{L_Y^\infty}\|\widehat{\FH}_n\|_{L^2}\right]\|\widehat{\FU}_n\|_{L^2}\\
&\leq\frac{\gamma_0\kappa}{8}\vep\left(\tn^2\|\widehat{\FH}_n\|_{L^2}^2+\|\pa_y\widehat{\FH}_n\|_{L^2}^2\right)+C\bar{M}(1+\bar{M})\|\widehat{\FW}_n\|_{L^2}^2.
\end{aligned}\end{align*}
By using $\hat{\psi}_{n}=\pa_y^{-1}\hat{h}_n$ and Hardy inequality, one has 
$\|y^{-1}\hat{\psi}_n\|_{L^2}\lesssim\|\hat{h}_n\|_{L^2},$ 
thus it follows from the bound on $\FD_\FU$ given in Lemma \ref{lm2.1} that
\begin{align*}
	\left|\left\langle \vep^{-\f12}\hat{\psi}_n\FD_\FU ,~{\widehat{\FU}}_n\right\rangle\right|\leq\|Y\FD_\FU\|_{L^\infty_Y}\|y^{-1}\hat{\psi}_n\|_{L^2}\|\widehat{\FU}_n\|_{L^2}\leq C\b{M}(1+\b{M})\|\widehat{\FW}_n\|_{L^2}^2.
\end{align*}
Similarly, one has
\begin{align*}
	\begin{aligned}
	&\left|\left\langle -2\kappa\vep^{-\f12}b_p\pa_y\widehat{\FH}_n+\bm{C}_\FH\widehat{\FH}_n+\vep^{-\f12}\hat{\psi}_n\FD_\FH,~G_s{\widehat{\FH}}_n\right\rangle \right|\leq \frac{\gamma_0\kappa}{8}\vep\|\pa_y\widehat{\FH}_n\|_{L^2}^2+C\bar{M}(1+\bar{M})\|\widehat{\FW}_n\|_{L^2}^2.
	\end{aligned}
\end{align*}
Also, it is easy to obtain
\begin{align*}
	\left|\left\langle \FR_{\FU,n},~{\widehat{\FU}}_n\right\rangle\right|+\left|\left\langle \FR_{\FH,n},~G_s{\widehat{\FH}}_n\right\rangle\right|\leq C\|\FR_n\|_{L^2}\|\widehat{\FW}_n\|_{L^2}.
\end{align*}
Plugging the above  estimates into \eqref{L2_y-0} yields
\begin{align*}
		\vep\left(\|\pa_y\widehat{\FW}_n\|_{L^2}^2+|\tn|^2\|\widehat{\FW}_n\|_{L^2}^2\right)\leq C\b{M}(1+\b{M})\|\widehat{\FW}_n\|_{L^2}^2+C\|\FR_n\|_{L^2}\|\|\widehat{\FW}_n\|_{L^2},
\end{align*} 
which implies the  estimate \eqref{2.6} and this completes the proof of the lemma.
\end{proof}

Next we establish a uniform-in-$\vep$ $L^2-$estimate on the velocity field $\widehat{\FU}_n$. 
\begin{lemma}\label{prop2.4}
	There exists a positive constant $C_4$ independent of $\vep$, $\tn$ and $\bar{M}$, such that
	\begin{align}\label{2.9}
	\begin{aligned}
	|\tn|^{\f12}\big\|\widehat{\FU}_n\big\|_{L^2}\leq&C_4\bar{M}^{\f12}(1+\b{M}^{\f12})\big\|\widehat{\FW}_n\big\|_{L^2}+C_4\|\FR_{n}\|_{L^2}^{\f12}\big\|\widehat{\FW}_n\big\|_{L^2}^{\f12}.
	\end{aligned}\end{align}
\end{lemma}
\begin{proof}
	From the second equation for $\widehat{\FH}_n$ in \eqref{pr-new-n}, we write
	\begin{align*}
		\displaystyle	-i\tn\widehat{\FU}_n=\ka\vep(\pa_y^2-\tn^2)\widehat{\FH}_n+2\kappa\varepsilon^{\f12}b_p\partial_y\widehat{\FH}_n-\bm{C}_{\FH}\widehat{\FH}_n-\varepsilon^{-\f12}\hat{\psi}_n\FD_{\FH}+\FR_{\FH,n}.
	\end{align*}
Then, taking inner product of the above equality with $-\widehat{\FU}_n$ yields 
\begin{align}\label{L2_u-0}
\begin{aligned}
i\tn\big\|\widehat{\FU}_n\big\|_{L^2}^2=&-\kappa\vep \left\langle(\pa_y^2-\tn^2)\widehat{\FH}_n,~\widehat{\FU}_n\right\rangle+\left\langle 2\kappa\varepsilon^{\f12}b_p\partial_y\widehat{\FH}_n-\bm{C}_{\FH}\widehat{\FH}_n-\varepsilon^{-\f12}\hat{\psi}_n\FD_{\FH},~\widehat{\FU}_n\right\rangle\\
&+\left\langle\FR_{\FH,n},~\widehat{\FU}_n\right\rangle.
\end{aligned}\end{align}
We estimate the right-hand side of \eqref{L2_u-0} term by term. First, by integration by part and the boundary condition $\widehat{\FU}_n|_{y=0}=\mathbf{0}$,  it is easy to get
\begin{align*}
	\kappa\vep \left|\left\langle(\pa_y^2-\tn^2)\widehat{\FH}_n,~\widehat{\FU}_n\right\rangle\right|\leq \ka\vep\left(\big\|\pa_y\widehat{\FW}_n\big\|_{L^2}^2+\tilde{n}^2\big\|\widehat{\FW}_n\big\|_{L^2}^2\right).
\end{align*}
Second, it follows by Cauchy-Schwarz inequality and Hardy inequality that
\begin{align*}
\begin{aligned}
	&\left|\left\langle 2\kappa\varepsilon^{\f12}b_p\partial_y\widehat{\FH}_n-\bm{C}_{\FH}\widehat{\FH}_n-\varepsilon^{-\f12}\hat{\psi}_n\FD_{\FH},~\widehat{\FU}_n\right\rangle\right|\\
	&\leq\left(2\kappa\sqrt{\varepsilon}\|b_p\|_{L_Y^\infty}\|\pa_Y\widehat{\FH}_n\|_{L^2}+\|\bm{C}_{\FH}\|_{L^\infty_Y}\|\widehat{\FH}_n\|_{L^2}+\|Y\mathbf{D}_{\FH}\|_{L^\infty_Y}\|y^{-1}\hat{\psi}_{n}\|_{L^2}\right)\|\widehat{\FU}_n\|_{L^2}\\
	&\leq C\varepsilon\big\|\pa_Y\widehat{\FH}_n\big\|_{L^2}^2+C\bar{M}(1+\bar{M})\big\|\widehat{\FW}_n\big\|_{L^2}^2,
\end{aligned}\end{align*}
where we have used \eqref{R} in the  last inequality. Note that
\begin{align*}
	\left|\left\langle\FR_{\FH,n},~\widehat{\FU}_n\right\rangle\right|\leq \|\FR_{\FH,n}\|_{L^2}\big\|\widehat{\FU}_n\big\|_{L^2}.
\end{align*}
Hence, we apply the above three inequalities to \eqref{L2_u-0} and obtain
\begin{align*}
	|\tn|~\big\|\widehat{\FU}_n\big\|_{L^2}^2\leq C\varepsilon\left(\big\|\pa_y\widehat{\FW}_n\big\|_{L^2}^2+\tilde{n}^2\big\|\widehat{\FW}_n\big\|_{L^2}^2\right)+C\bar{M}(1+\bar{M})\big\|\widehat{\FW}_n\big\|_{L^2}^2+\|\FR_{\FH,n}\|_{L^2}\big\|\widehat{\FU}_n\big\|_{L^2}, 
\end{align*}
which, along with \eqref{2.6}, gives the estimate \eqref{2.9} and this completes the proof of the lemma.
\end{proof}

Next we turn to  the $L^2$-estimate of $\widehat{\FH}_n$. We point out  that if we estimate $\big\|\widehat{\FH}_n\big\|_{L^2}$ in a similar way as Lemma \ref{prop2.4}, a boundary term $\hat{h}_n\pa_y\hat{u}_n|_{y=0}$ appears due to the mix boundary condition \eqref{1.1-1}. Clearly, it is impossible
 to control this term with the low regularity of the solution. In order to overcome this difficulty, in what follows we turn to establish a weighted estimate on $\widehat{\FH}_n$ with the weight $Z^{\f12}(y)$. Notice that the function $Z(y)$ depends on the variable $y$. To avoid the commutator with pressure term $P$, we use the vorticity formulation. Let $\omega_u=\pa_y\hat{u}-\pa_x\hat{v}$ and $\omega_h=\pa_y\hat{h}-\pa_x\hat{g}$ be the vorticity of $(\hat{u},\hat{v})$ and $(\hat{h},\hat{g})$ respectively. We recall that $\hat{\psi}$ is the stream function of $\widehat{\FH}$ and denote by $\hat{\phi}$ the stream function of $\widehat{\FU}$, {\it i.e.},
$$\hat{\phi}_y=\hat{u},\quad -\hat{\phi}_x=\hat{v},\quad \hat{\phi}|_{y=0}=0.
$$
We denote respectively by $\omega_{u,n}$ and $\omega_{h,n}$ the $n$-th Fourier coefficients of $\omega_u$ and $\omega_h$. Similarly, the $n$-th Fourier coefficient of $\hat{\phi}$ 
is denoted by $\hat{\phi}_n$. 
That is, 
\begin{align*}
\begin{aligned}
	\omega_{u,n}=curl~\widehat{\FU}_n=\pa_y\hat{u}_n-i\tn\hat{v}_n=(\pa_y^2-\tn^2)\hat{\phi}_n,\quad
	\omega_{h,n}=curl~\widehat{\FH}_n=\pa_y\hat{h}_n-i\tn\hat{g}_n=(\pa_y^2-\tn^2)\hat{\psi}_n.
\end{aligned}\end{align*}
From the system \eqref{pr-new-n} for $\widehat{\FW}_n$, we use the second equation in \eqref{pr-new-n} to eliminate $\widehat{\FU}_n$ in the first equation. Then it holds
\begin{align}\label{vorticity-u}
\begin{aligned}
	&-i\tn G_s\widehat{\FH}_n+(i\tn p_n,\pa_yp_n)^T-\mu\vep(\pa_y^2-\tn^2)\widehat{\FU}_n-\varepsilon(\kappa+\mu) U_s(\pa_y^2-\tn^2)\widehat{\FH}_n\\
	&=\FR_{\FU,n}+\frac{\mu+\kappa}{\kappa}U_s\FR_{\FH,n}-i\tn\sqrt{\varepsilon}\bm{A}_{\FU}\widehat{\FH}_n-\sqrt{\varepsilon}\Big(\bm{B}_{\FU}-2(\kappa+\mu)a_p\pa_YH_s\bm{I}_{2}\Big)\pa_{y}\widehat{\FH}_n\\
	&\quad-\left(\bm{C}_{\FU}+\frac{\mu+\kappa}{\kappa}U_s\bm{C}_\FH\right)\widehat{\FH}_n-\varepsilon^{-\f12}\hat{\psi}_n\left(\FD_{\FU}+\frac{\mu+\kappa}{\kappa}U_s\FD_\FH\right)\\
	&\triangleq \widetilde{\FR}_{\FU,n},
\end{aligned}\end{align}
where $\bm{I}_{2}$ is the identity matrix of order 2.
By taking curl on the above equation and the second equation in \eqref{pr-new-n} respectively, we arrive at the following system for
$\omega_{u,n}$ and $\omega_{h,n}$:
\begin{equation}\label{2.11}\left\{
\begin{aligned}
-i\tilde{n}~curl\big(G_s\widehat{\FH}_n\big)-\vep\mu(\pa_y^2-\tilde{n}^2)\omega_{u,n}-\vep(\ka+\mu)~curl\left(U_s(\pa_y^2-\tn^2)\widehat{\FH}_n\right)&=curl~ \widetilde{\FR}_{\FU,n},\\
-i\tilde{n}\omega_{u,n}-\vep\ka(\pa_y^2-\tilde{n}^2)\omega_{h,n}&=curl~ \widetilde{\FR}_{\FH,n},
\end{aligned}\right.
\end{equation}
where 
$$
 \widetilde{\FR}_{\FH,n}=\FR_{\FH,n}+2\kappa\sqrt{\varepsilon}\pa_y\widehat{\FH}_n-\bm{C}_{\FH}\widehat{\FH}_n+\varepsilon^{-\f12}\hat{\psi}_n\FD_{\FH}.
$$

The weighted estimte on $\widehat{\FW}_n$ is given in the following lemma.

\begin{lemma}\label{prop2.5} For sufficient small $\vep$, there exists a positive constant $C_5$ independent of $\vep$, $\tn$ and $\b{M}$, such that for any $\eta>0$ and $\delta\geq0$, it holds that
	\begin{align}\label{2.12}
	\begin{aligned}
	|\tn|^{\f12}\|Z^{\f12}\widehat{\FW}_n\|_{L^2}\leq~&C_5 |\tn|^{-\f12}|\log\vep|^{1+\frac{\eta}{3}} \big\|Z^{\f12}{\FR}_{n}\big\|_{L^2}+C_5{\vep}^{\f14}\bar{M}^{\f12}\big\|{\FR}_{n}\big\|_{L^2}^{\f12}\big\|\widehat{\FW}_n\big\|_{L^2}^{\f12}\\
	&+C_5{\vep}^{\f14}(1+{\bar{M}}^{\f12})\big\|{\FR}_{n}\big\|_{L^2}^{\f14}\big\|\widehat{\FW}_n\big\|_{L^2}^{\f34}+C_5{\vep}^{\f14}
	\bar{M}^{\f14}(1+\bar{M}^{\f54})\big\|\widehat{\FW}_n\big\|_{L^2}.\\
	\end{aligned}\end{align}
\end{lemma}
\begin{proof}
  By taking inner product of the first and second equations in \eqref{2.11} with $sgn(\tn)\mu^{-1}Z{\hat{\psi}_n}$ and $sgn(\tn)\ka^{-1}Z{\hat{\phi}_n}$ respectively and adding them together, then taking its imaginary part, we obtain
	\begin{align}\label{2.13}
	\sum_{i=1}^4I_i=0,
	\end{align}
	where 
	$$\begin{aligned}
	I_1&=-|\tn|~\text{Re}\left(\big\langle curl\big(G_s\widehat{\FH}_n\big),~\mu^{-1}Z{\hat{\psi}_n} \big\rangle+\big\langle \omega_{u,n},~\ka^{-1}Z{\hat{\phi}_n}\big\rangle\right),\\
	I_2&=-\vep sgn(\tn)~\text{Im}\left(\big\langle  (\pa_y^2-\tilde{n}^2)\omega_{u,n},~Z{\hat{\psi}_n} \big\rangle +\big\langle  (\pa_y^2-\tilde{n}^2)\omega_{h,n},~Z{\hat{\phi}_n} \big\rangle\right),\\
	I_3&=-\vep sgn(\tn)~\text{Im}\big\langle
	(\mu+\kappa)curl\left(U_s(\pa_y^2-\tilde{n}^2)\widehat{\FH}_n\right),~\mu^{-1}Z{\hat{\psi}_n}  \big\rangle,\\
	I_4&=-\frac{sgn(\tn)}{\mu}\text{Im}\big\langle curl~ \widetilde{\FR}_{\FU,n},~Z{\hat{\psi}_n}\big\rangle-\frac{sgn(\tn)}{\kappa}\text{Im}\big\langle curl~ \widetilde{\FR}_{\FH,n},~Z{\hat{\phi}_n}\big\rangle.
	\end{aligned}
	$$
We estimate $I_1$ to $I_4$ term by term. For $I_1$, by integration by parts and using the boundary condition $\hat{\psi}_n|_{y=0}=\hat{\phi}_n|_{y=0}=0$, it holds
\begin{align*}
\begin{aligned}
I_1=&|\tn|\left(\mu^{-1}\big\|\sqrt{G_sZ}~\widehat{\FH}_n\big\|_{L^2}^2+\kappa^{-1}\big\|\sqrt{Z}~\widehat{\FU}_n\big\|_{L^2}^2\right)+|\tn|\text{Re}\left(\mu^{-1}\big\langle G_s\hat{h}_n,~\pa_yZ~\hat{\psi}_n\big\rangle+\kappa^{-1}\big\langle \hat{u}_n,~\pa_yZ~\hat{\phi}_n\big\rangle \right)\\
:=&I_{1,1}+I_{1,2}.
\end{aligned}\end{align*}
From \eqref{A4}, it follows
$$I_{1,1}\geq c_0|\tilde{n}|~\big\|Z^{\f12}\widehat{\FW}_n\big\|_{L^2}^2
$$
for some positive constants $c_0$ independent of $\vep$ and $\b{M}$. For $I_{1,2}$,  we  notice that $\hat{h}_n=\pa_y\hat{\psi}_n, \hat{u}_n=\pa_y\hat{\phi}_n$ and $\pa_y Z\equiv0, ~y\geq2$. Then
 by integration by parts and using boundary condition $\hat{\psi}_n|_{y=0}=\hat{\phi}_n|_{y=0}=0$, we 
 can write it as
 \begin{align*}
 	\begin{aligned}
 	I_{1,2}&=\frac{|\tn|}{2}\int_0^2\pa_yZ\left[\mu^{-1}G_s\pa_y(|\hat{\psi}_n|^2)+\kappa^{-1}\pa_y(|\hat{\phi}_n|^2)\right]dy\\
 	&=-\frac{|\tn|}{2\mu}\int_0^2\pa_y\left(G_s\pa_yZ\right)|\hat{\psi}_n|^2dy-\frac{|\tn|}{2\kappa}\int_0^2\pa_y^2Z|\hat{\phi}_n|^2dy.
 	\end{aligned}
 \end{align*}
 We estimate the two terms on the right-hand side of the above identity. For the first one, 
 we use \eqref{z2} and \eqref{z4} to obtain
 $$
 \begin{aligned}
 -\frac{|\tilde{n}|}{2\mu}\int_0^2\pa_y\left(G_s\pa_yZ\right)|\hat{\psi}_n|^2d y&=-\frac{|\tilde{n}|}{2\mu}\int_1^2\pa_y\left(G_s\pa_yZ\right)|\hat{\psi}_n|^2dy\\
 &\geq -C\bar{M}\vep|\tilde{n}|\int_1^2 y^{-2}|\hat{\psi}_n|^2dy\geq -C\bar{M}\vep|\tilde{n}|~\|\hat{h}_n\|_{L^2}^2,
 \end{aligned}
 $$
 where we have used Hardy inequality in the last inequality.
 Similarly, the second term is bounded from below as
 $$
 \begin{aligned}
 -\frac{|\tilde{n}|}{2\kappa}\int_0^2\pa_y^2Z|\hat{\phi}_n|^2dy
 &=-\frac{|\tilde{n}|}{2\kappa}\left(\int_0^{3/2}\pa_y^2Z~|\hat{\phi}_n|^2dy+\int_{3/2}^{2}\pa_y^2Z~|\hat{\phi}_n|^2dy\right)\\
 &\geq-\frac{|\tilde{n}|}{2\kappa}\int_0^{3/2}|y^2\pa_y^2Z|\cdot \left|\frac{\hat{\phi}_n}{y}\right|^2dy\geq -C\sqrt{\vep}\b{M}|\tilde{n}|\|\hat{u}_n\|_{L^2}^2.
 \end{aligned}
 $$
 By combining the above estimates related to $I_1$, one has
\begin{align}\label{2.15}
\begin{aligned}
I_1&\geq 
c_0|\tilde{n}|~\|Z^{\f12}\widehat{\FW}_n\|_{L^2}^2  -C\bar{M}|\tilde{n}|\left(\vep\|\hat{h}_n\|_{L^2}^2+\sqrt{\vep}\|\hat{u}_n\|_{L^2}^2\right).
\end{aligned}\end{align}

Next we consider $I_2$. The boundary conditions
$Z|_{y=0}=\hat{\phi}_n|_{y=0}=\hat{\psi}_n|_{y=0}=0$ allow us to use integration by parts twice. That is, we have
$$
\begin{aligned}
I_2=&-sgn(\tn)\vep\text{Im}\Big(\langle  \omega_{u,n},~Z{\omega}_{h,n}\rangle+ \langle  \omega_{h,n},~Z{\omega}_{u,n}\rangle\Big)-2sgn(\tn)\vep\text{Im}\Big(\langle  \omega_{u,n},~\pa_yZ{\hat{h}_n}\rangle+ \langle  \omega_{h,n},~\pa_yZ{\hat{u}_n}\rangle\Big)\\
&-sgn(\tn)\vep\text{Im}\Big(\langle  \omega_{u,n},~\pa_y^2Z{\hat{\psi}_n}\rangle+ \langle  \omega_{h,n},~\pa_y^2Z{\hat{\phi}_n}\rangle\Big)\\
:=&I_{2,1}+I_{2,2}+I_{2,3}.
\end{aligned}
$$
It is straightforward to see that $I_{2,1}=0.$ 
By the Cauchy-Schwarz inequality, 
$$
\begin{aligned}
|I_{2,2}|&\leq 2\vep\left|\langle  \omega_{u,n},~\pa_yZ{\hat{h}_n}\rangle+ \langle  \omega_{h,n},\pa_yZ{\hat{u}_n}\rangle\right|\\
&\leq C\vep\|\pa_yZ\|_{L^\infty}\left(\|\omega_{u,n}\|_{L^2}\|\hat{h}_n\|_{L^2}+\|\omega_{h,n}\|_{L^2}\|\hat{u}_n\|_{L^2}\right)\\
&\leq C\vep\left(\|\pa_y\widehat{\FW}_n\|_{L^2}+|\tilde{n}|\|\widehat{\FW}_n\|_{L^2}\right)\|\widehat{\FW}_n\|_{L^2}.
\end{aligned}
$$
And by  the Hardy inequality and \eqref{z5}, 
$$
\begin{aligned}
|I_{2,3}|&\leq\vep\left|\langle  \omega_{u,n},~\pa_y^2Z{\hat{\psi}_n}\rangle+\langle  \omega_{h,n},~\pa_y^2Z{\hat{\phi}_n}\rangle\right|\\
&\leq C\vep\|y\pa_y^2Z\|_{L^\infty}\left(\|\omega_{u,n}\|_{L^2}\|y^{-1}\hat{\psi}_n\|_{L^2}+\|\omega_{h,n}\|_{L^2}\|y^{-1}\hat{\phi}_n\|_{L^2}\right)\\
&\leq  C(1+\bar{M})\vep\left(\|\pa_y\widehat{\FW}_n\|_{L^2}+|\tilde{n}|\|\widehat{\FW}_n\|_{L^2}\right)\|\widehat{\FW}_n\|_{L^2}.
\end{aligned}
$$
Hence combining the above three estimates yields
\begin{align}\label{2.16}
|I_2|\leq C(1+\b{M})\vep\left(\|\pa_y\widehat{\FW}_n\|_{L^2}+|\tilde{n}|\|\widehat{\FW}_n\|_{L^2}\right)\|\widehat{\FW}_n\|_{L^2}.
\end{align}
The term $I_3$ can be treated in a similar way. In fact, by integration by parts and the boundary conditions
$Z|_{y=0}=\hat{\psi}_n|_{y=0}=0$, we have
$$
\begin{aligned}
I_3=&-\vep sgn(\tn)~\big(1+\frac{\kappa}{\mu}\big)\text{Im}\left\langle \pa_y\hh,(2\pa_yZU_s+Z\pa_yU_s)\hh+\pa_y(U_s\pa_yZ)\hat{\psi}_n  \right\rangle.
\end{aligned}
$$
Note that from \eqref{z5}, it holds
\begin{align}\label{useful}
\begin{aligned}
|\pa_yZU_s|&\leq \|\pa_yZ\|_{L^\infty}\|U_s\|_{L^\infty_Y}\leq C,\quad |Z\pa_yU_s|\leq \|y^{-1}Z\|_{L^\infty}\|Y\pa_YU_s\|_{L^\infty_Y}\leq C(1+\b{M}),
\end{aligned}
\end{align}
and
$$\begin{aligned}
|y\pa_y(U_s\pa_yZ)|&\leq |y\pa_y^2Z~U_s|+\vep^{-\f12}|y\pa_yZ\pa_YU_s|\leq \|y\pa_y^2Z\|_{L^\infty}\|U_s\|_{L^\infty_Y}+\|\pa_yZ\|_{L^\infty}\|Y\pa_YU_s\|_{L^\infty_Y}\\
&\leq C(1+\b{M}).
\end{aligned}
$$
Thus one has
\begin{align}\label{2.17}
\begin{aligned}
|I_3|&\leq C\vep\|\pa_y\hh\|_{L^2}\left((\|\pa_yZU_s\|_{L^\infty}+\|Z\pa_yU_s\|_{L^\infty})\|\hat{h}_n\|_{L^2}+\|y\pa_y(U_s\pa_yZ)\|_{L^\infty} \|y^{-1}\hat{\psi}_{n}\|_{L^2}\right)\\
&\leq C(1+\b{M})\vep\left(\big\|\pa_y\widehat{\FW}_n\big\|_{L^2}+|\tilde{n}|\big\|\widehat{\FW}_n\big\|_{L^2}\right)\big\|\widehat{\FH}_n\big\|_{L^2}.
\end{aligned}\end{align}

For $I_4$, we first estimate $\left|\text{Im}\left\langle curl~ \widetilde{\FR}_{\FU,n}, ~Z{\hat{\psi}_n}\right\rangle \right|$. By integration by parts, 
\begin{align}\label{est_I4}
	\begin{aligned}
	\left\langle curl~ \widetilde{\FR}_{\FU,n}, ~Z{\hat{\psi}_n}\right\rangle=-\left\langle  \widetilde{\FR}_{\FU,n}, ~Z{\widehat{\FH}_n}\right\rangle-\left\langle  \tilde{R}_{u,n}, ~\pa_yZ{\hat{\psi}_n}\right\rangle,
	\end{aligned}
\end{align}
where  $\widetilde{\FR}_{\FU,n}=(\tilde{R}_{u,n},\tilde{R}_{v,n}).$ As from \eqref{vorticity-u}, it holds
\begin{align*}
	\begin{aligned}
	\widetilde{\FR}_{\FU,n}=&\FR_{\FU,n}+\frac{\mu+\kappa}{\kappa}U_s\FR_{\FH,n}-i\tn\sqrt{\varepsilon}\bm{A}_{\FU}\widehat{\FH}_n-\sqrt{\varepsilon}\Big(\bm{B}_{\FU}-2(\kappa+\mu)a_p\pa_YH_s\bm{I}_{2}\Big)\pa_{y}\widehat{\FH}_n\\
	&-\left(\bm{C}_{\FU}+\frac{\mu+\kappa}{\kappa}U_s\bm{C}_\FH\right)\widehat{\FH}_n-\varepsilon^{-\f12}\hat{\psi}_n\left(\FD_{\FU}+\frac{\mu+\kappa}{\kappa}U_s\FD_\FH\right).
	\end{aligned}
\end{align*}
Hence, we have
\begin{align}\label{est-I4-1}
\begin{aligned}
	&\left|\left\langle  \widetilde{\FR}_{\FU,n}, ~Z{\widehat{\FH}_n}\right\rangle\right|\\
	&\lesssim \big\|Z^{\f12}{\FR}_{n}\big\|_{L^2}\big\|Z^{\f12}\widehat{\FH}_n\big\|_{L^2}+\|y^{-1}Z\|_{L^\infty}\big\|\widehat{\FH}_n\big\|_{L^2}\cdot\left\{\vep|\tn|\big\|Y\bm{A}_{\FU}\big\|_{L^\infty_Y}\big\|\widehat{\FH}_n\big\|_{L^2}\right.\\
	&\qquad+\vep\left(\big\|Y\bm{B}_{\FU}\big\|_{L^\infty_Y}+\|Y\pa_YH_s\|_{L^\infty_Y}\right)\big\|\pa_y\widehat{\FH}_n\big\|_{L^2}+\sqrt{\vep}\left(\big\|Y\bm{C}_{\FU}\|_{L^\infty_Y}+\big\|Y\bm{C}_{\FH}\|_{L^\infty_Y}\right)\big\|\widehat{\FH}_n\big\|_{L^2}\\
	&\qquad\left.+\sqrt{\vep}\left(\big\|Y^2\bf{D}_{\FU}\|_{L^\infty_Y}+\big\|Y^2\bf{D}_{\FH}\|_{L^\infty_Y}\right)\big\|y^{-1}\hat{\psi}_n\big\|_{L^2}\right\}\\
	&\lesssim\big\|Z^{\f12}{\FR}_{n}\big\|_{L^2}\big\|Z^{\f12}\widehat{\FH}_n\big\|_{L^2}+\vep\b{M}\left(\big\|\pa_y\widehat{\FH}_n\big\|_{L^2}+|\tilde{n}|\big\|\widehat{\FH}_n\big\|_{L^2}\right)\big\|\widehat{\FH}_n\big\|_{L^2}+\sqrt{\vep}\bar{M}(1+\bar{M})\big\|\widehat{\FH}_n\big\|_{L^2}^2.
\end{aligned}\end{align}
Similarly, by noting that $\hat{\psi}_n=\pa_y^{-1}\hat{h}_n$,  it holds that for any $\eta>0$ and $\delta\geq0,$
\begin{align}\label{est-I4-2}
\begin{aligned}
&\left|\left\langle  \tilde{R}_{u,n}, ~\pa_yZ{\hat{\psi}_n}\right\rangle\right|\\
&\lesssim|\log\vep|^{1+\frac{\eta}{3}}\|\pa_yZ\|_{L^\infty} \big\|Z^{\f12}{\FR}_{n}\big\|_{L^2}\left(\big\|Z^{\f12}\hat{h}_n\big\|_{L^2}+\vep^{\frac{1}{4}+\delta}\big\|\hat{h}_n\big\|_{L^2}\right)\\
&\quad+\|\pa_yZ\|_{L^\infty}\big\|y^{-1}\hat{\psi}_n\big\|_{L^2}\left\{\vep|\tn|\big\|Y\bm{A}_{\FU}\big\|_{L^\infty_Y}\big\|\widehat{\FH}_n\big\|_{L^2}+\vep\left(\big\|Y\bm{B}_{\FU}\big\|_{L^\infty_Y}+\|Y\pa_YH_s\|_{L^\infty_Y}\right)\big\|\pa_y\widehat{\FH}_n\big\|_{L^2}\right.\\
&\qquad\qquad+\sqrt{\vep}\left(\big\|Y\bm{C}_{\FU}\|_{L^\infty_Y}+\big\|Y\bm{C}_{\FH}\|_{L^\infty_Y}\right)\big\|\widehat{\FH}_n\big\|_{L^2}\left.+\sqrt{\vep}\left(\big\|Y^2\bf{D}_{\FU}\|_{L^\infty_Y}+\big\|Y^2\bf{D}_{\FH}\|_{L^\infty_Y}\right)\big\|y^{-1}\hat{\psi}_n\big\|_{L^2}\right\}\\
&\lesssim|\log\vep|^{1+\frac{\eta}{3}} \big\|Z^{\f12}{\FR}_{n}\big\|_{L^2}\left(\big\|Z^{\f12}\hat{h}_n\big\|_{L^2}+\vep^{\frac{1}{4}+\delta}\big\|\hat{h}_n\big\|_{L^2}\right)\\
&\quad+\vep\b{M}\left(\big\|\pa_y\widehat{\FH}_n\big\|_{L^2}+|\tilde{n}|\big\|\widehat{\FH}_n\big\|_{L^2}\right)\big\|\widehat{\FH}_n\big\|_{L^2}+
\sqrt{\vep}\bar{M}(1+\bar{M})\big\|\widehat{\FH}_n\big\|_{L^2}^2,
\end{aligned}\end{align}
where we have used \eqref{w1_1} to obtain the first term on the right-hand side of the  first inequality. Applying \eqref{est-I4-1} and \eqref{est-I4-2} to \eqref{est_I4} yields
\begin{align}\label{est-I4-3}
	\begin{aligned}
	&\left|\left\langle curl~ \widetilde{\FR}_{\FU,n}, ~Z{\hat{\psi}_n}\right\rangle\right|\\
	&\lesssim|\log\vep|^{1+\frac{\eta}{3}} \big\|Z^{\f12}{\FR}_{n}\big\|_{L^2}\left(\big\|Z^{\f12}\widehat{\FH}_n\big\|_{L^2}+\vep^{\frac{1}{4}+\delta}\big\|\hat{h}_n\big\|_{L^2}\right)\\
	&\quad+\vep\b{M}\left(\big\|\pa_y\widehat{\FH}_n\big\|_{L^2}+|\tilde{n}|\big\|\widehat{\FH}_n\big\|_{L^2}\right)\big\|\widehat{\FH}_n\big\|_{L^2}+
	\sqrt{\vep}\bar{M}(1+\bar{M})\big\|\widehat{\FH}_n\big\|_{L^2}^2\\
	&\leq~\frac{c_0|\tn|}{4}\big\|Z^{\f12}\widehat{\FH}_n\big\|_{L^2}^2+C|\tn|^{-1}|\log\vep|^{2+\frac{2\eta}{3}} \big\|Z^{\f12}{\FR}_{n}\big\|_{L^2}^2+C\vep^{\f14+\delta}|\log\vep|^{1+\frac{\eta}{3}} \big\|Z^{\f12}{\FR}_{n}\big\|_{L^2}\|\hat{h}_n\|_{L^2}\\
	&\quad+C\vep\b{M}\left(\big\|\pa_y\widehat{\FH}_n\big\|_{L^2}+|\tilde{n}|\big\|\widehat{\FH}_n\big\|_{L^2}\right)\big\|\widehat{\FH}_n\big\|_{L^2}+C\sqrt{\vep}\bar{M}(1+\bar{M})\big\|\widehat{\FH}_n\big\|_{L^2}^2.
	\end{aligned}
\end{align}
Similarly,  one can obtain
\begin{align*}
\begin{aligned}
&\left|\left\langle curl~ \widetilde{\FR}_{\FH,n}, ~Z{\hat{\phi}_n}\right\rangle\right|\\
&\leq~\frac{c_0|\tn|}{4}\big\|Z^{\f12}\widehat{\FU}_n\big\|_{L^2}^2+C|\tn|^{-1}|\log\vep|^{2+\frac{2\eta}{3}} \big\|Z^{\f12}{\FR}_{n}\big\|_{L^2}^2+C\vep^{\f14+\delta}|\log\vep|^{1+\frac{\eta}{3}} \big\|Z^{\f12}{\FR}_{n}\big\|_{L^2}\|\hat{u}_n\|_{L^2}\\
&\quad+C\vep\b{M}\big\|\pa_y\widehat{\FH}_n\big\|_{L^2}\big\|\widehat{\FU}_n\big\|_{L^2}+C\sqrt{\vep}\bar{M}(1+\bar{M})\big\|\widehat{\FH}_n\big\|_{L^2}\big\|\widehat{\FU}_n\big\|_{L^2}.
\end{aligned}
\end{align*}
Then combining the above two estimates gives
\begin{align}\label{est-I4}
\begin{aligned}
	|I_4|\leq~&\frac{c_0|\tn|}{2}\big\|Z^{\f12}\widehat{\FW}_n\big\|_{L^2}^2
	+C|\tn|^{-1}|\log\vep|^{2+\frac{2\eta}{3}} \big\|Z^{\f12}{\FR}_{n}\big\|_{L^2}^2+C\vep^{\f14+\delta}|\log\vep|^{1+\frac{\eta}{3}}\big\|Z^{\f12}{\FR}_{n}\big\|_{L^2}\|\widehat{\FW}_n\|_{L^2}\\
&+C\vep\b{M}\left(\big\|\pa_y\widehat{\FH}_n\big\|_{L^2}+|\tilde{n}|\big\|\widehat{\FH}_n\big\|_{L^2}\right)\big\|\widehat{\FW}_n\big\|_{L^2}+C
\sqrt{\vep}\bar{M}(1+\bar{M})\big\|\widehat{\FW}_n\big\|_{L^2}^2. 
\end{aligned}\end{align}

Thus we complete the estimates of $I_1-I_4$. By substituting \eqref{2.15}, \eqref{2.16}, \eqref{2.17} and \eqref{est-I4} into \eqref{2.13}, 
we obtain 
	\begin{align}\label{est-wL2}
	\begin{aligned}
|\tn|~\|Z^{\f12}\widehat{\FW}_n\|_{L^2}^2\lesssim~&\sqrt{\vep}\bar{M}|n|\|\hat{u}_n\|_{L^2}^2+\vep(1+\b{M})\left(\|\pa_y\widehat{\FW}_n\|_{L^2}+|\tilde{n}|	\|\widehat{\FW}_n\|_{L^2}
\right)\|\widehat{\FW}_n\|_{L^2}\\
&+|\tn|^{-1}|\log\vep|^{2+\frac{2\eta}{3}} \big\|Z^{\f12}{\FR}_{n}\big\|_{L^2}^2+\vep^{\f14+\delta}|\log\vep|^{1+\frac{\eta}{3}} \big\|Z^{\f12}{\FR}_{n}\big\|_{L^2}\|\widehat{\FW}_n\|_{L^2}\\
&+\sqrt{\vep}\bar{M}(1+\bar{M})\big\|\widehat{\FW}_n\big\|_{L^2}^2.
\end{aligned}\end{align}
From \eqref{2.6} and \eqref{2.9} one has
\begin{align*}
		\sqrt{\vep}\left(\|\pa_y\widehat{\FW}_n\|_{L^2}+|\tn|\|\widehat{\FW}_n\|_{L^2}\right)\lesssim\b{M}^{\f12}(1+\b{M}^{\f12})\|\widehat{\FW}_n\|_{L^2}+\|\FR_n\|_{L^2}^{\f12}\|\widehat{\FW}_n\|_{L^2}^{\f12},
\end{align*}
and
\begin{align*}
	|n|\|\hat{u}_n\|_{L^2}^2\lesssim \bar{M}(1+\bar{M})\big\|\widehat{\FW}_n\big\|_{L^2}^2+\|\FR_{n}\|_{L^2}\big\|\widehat{\FW}_n\big\|_{L^2}.
\end{align*}
Then substituting the above two inequalities into \eqref{est-wL2} implies
	\begin{align*}
\begin{aligned}
|\tn|~\|Z^{\f12}\widehat{\FW}_n\|_{L^2}^2\lesssim~&\sqrt{\vep}\bar{M}\left[\bar{M}(1+\bar{M})\big\|\widehat{\FW}_n\big\|_{L^2}^2+\|\FR_{n}\|_{L^2}\big\|\widehat{\FW}_n\big\|_{L^2}\right]\\
&+\sqrt{\vep}(1+\b{M})\|\widehat{\FW}_n\|_{L^2}\left(\b{M}^{\f12}(1+\b{M}^{\f12})\|\widehat{\FW}_n\|_{L^2}+\|\FR_n\|_{L^2}^{\f12}\|\widehat{\FW}_n\|_{L^2}^{\f12}
\right)\\
&+|\tn|^{-1}|\log\vep|^{2+\frac{2\eta}{3}} \big\|Z^{\f12}{\FR}_{n}\big\|_{L^2}^2+\vep^{\f14+\delta}|\log\vep|^{1+\frac{\eta}{3}} \big\|Z^{\f12}{\FR}_{n}\big\|_{L^2}\|\widehat{\FW}_n\|_{L^2}\\
&+\sqrt{\vep}\bar{M}(1+\bar{M})\big\|\widehat{\FW}_n\big\|_{L^2}^2\\
\lesssim~&|\tn|^{-1}|\log\vep|^{2+\frac{2\eta}{3}} \big\|Z^{\f12}{\FR}_{n}\big\|_{L^2}^2+\sqrt{\vep}\bar{M}\big\|{\FR}_{n}\big\|_{L^2}\big\|\widehat{\FW}_n\big\|_{L^2}\\
&+\sqrt{\vep}(1+\bar{M})\big\|{\FR}_{n}\big\|_{L^2}^{\f12}\big\|\widehat{\FW}_n\big\|_{L^2}^{\f32}+
\sqrt{\vep}\bar{M}^{\f12}(1+\bar{M}^{\f52})\big\|\widehat{\FW}_n\big\|_{L^2}^2,
\end{aligned}\end{align*}
 provided $\varepsilon$ small enough. Hence we obtain \eqref{2.12}  and then complete the proof of the lemma.
\end{proof}

To recover the $L^2$-estimate of $\widehat{\FW}_n$ by the interpolation inequality \eqref{AP1},  we have the following lemma.

\begin{lemma}
	There exist positive constants  $\delta_1$ and $\vep_1$, such that if
$$ 
\varrho(\b{M}+\b{M}^4)\leq \delta_1,\qquad \vep\in (0,\vep_1),
$$ 
then 
	\begin{align}\label{2.18}
	\begin{aligned}
	&\sqrt{\vep}|\tn|^{\f13}\left(\|\pa_y\widehat{\FW}_n\|_{L^2}+|\tn|\big\|\widehat{\FW}_n\big\|_{L^2}\right)+|\tn|^{\f23}\big\|\widehat{\FW}_n\big\|_{L^2}+\vep^{-\f14}|\tn|^{\f56}\big\|Z^{\f12}\widehat{\FW}_n\big\|_{L^2}\\
	&\leq C_6 |\log\vep|^{1+\frac{\eta}{3}}\left(\big\|{\FR}_{n}\big\|_{L^2}+\vep^{-\f14} \big\|Z^{\f12}{\FR}_{n}\big\|_{L^2}\right),
	\end{aligned}\end{align}
	where the positive constant $C_6$ is independent of $\vep$ and $n$.
\end{lemma}
\begin{proof}
	We apply the estimates \eqref{2.6} and \eqref{2.12} to the interpolation inequality \eqref{AP1} for $\widehat{\FW}_n$, and obtain
	\begin{align}\label{L2-0}
		\begin{aligned}
		&\big\|\widehat{\FW}_n\big\|_{L^2}\leq~ 2\sqrt{2C_0}\|Z^{\f12}\widehat{\FW}_n\|_{L^2}^{\f23}\|\pa_y\widehat{\FW}_n\|_{L^2}^{\f13}+C_0\|Z^{\f12}\widehat{\FW}_n\|_{L^2}\\
		&\leq 2\sqrt{2C_0}|\tn|^{-\f13}C_5^{\f23}\left\{|\tn|^{-\f12}|\log\vep|^{1+\frac{\eta}{3}} \big\|Z^{\f12}{\FR}_{n}\big\|_{L^2}+{\vep}^{\f14}\bar{M}^{\f12}\big\|{\FR}_{n}\big\|_{L^2}^{\f12}\big\|\widehat{\FW}_n\big\|_{L^2}^{\f12}\right.\\
		&\hspace{3.5cm}\left.+{\vep}^{\f14}(1+{\bar{M}}^{\f12})\big\|{\FR}_{n}\big\|_{L^2}^{\f14}\big\|\widehat{\FW}_n\big\|_{L^2}^{\f34}+{\vep}^{\f14}
		\bar{M}^{\f14}(1+\bar{M}^{\f54})\big\|\widehat{\FW}_n\big\|_{L^2}\right\}^{\f23}\\
		&\qquad\qquad\cdot\vep^{-\f16}C_3^{\f13}\left\{\b{M}^{\f12}(1+\b{M}^{\f12})\|\widehat{\FW}_n\|_{L^2}+\|\FR_n\|_{L^2}^{\f12}\|\widehat{\FW}_n\|_{L^2}^{\f12}\right\}^{\f13}\\
		&\quad+C_0|\tn|^{-\f12}C_5\left\{|\tn|^{-\f12}|\log\vep|^{1+\frac{\eta}{3}} \big\|Z^{\f12}{\FR}_{n}\big\|_{L^2}+{\vep}^{\f14}\bar{M}^{\f12}\big\|{\FR}_{n}\big\|_{L^2}^{\f12}\big\|\widehat{\FW}_n\big\|_{L^2}^{\f12}\right.\\
		&\hspace{3cm}\left.+{\vep}^{\f14}(1+{\bar{M}}^{\f12})\big\|{\FR}_{n}\big\|_{L^2}^{\f14}\big\|\widehat{\FW}_n\big\|_{L^2}^{\f34}+{\vep}^{\f14}
		\bar{M}^{\f14}(1+\bar{M}^{\f54})\big\|\widehat{\FW}_n\big\|_{L^2}\right\}\\
		&\leq\left[\f14+2\sqrt{2C_0}C_3^{\f13}C_5^{\f23}|\tn|^{-\f13}\bar{M}^{\f13}(1+\bar{M})+C_0C_5|\tn|^{-\f12}\bar{M}^{\f14}(1+\bar{M}^{\f54}){\vep}^{\f14}\right]\big\|\widehat{\FW}_n\big\|_{L^2}\\
		&\quad+C|\tn|^{-\f23}\big(1+|\tn|^{-\f43}\big)|\log\vep|^{1+\frac{\eta}{3}}\left(\vep^{-\f14} \big\|Z^{\f12}{\FR}_{n}\big\|_{L^2}+\big\|{\FR}_{n}\big\|_{L^2}\right),
		\end{aligned}
	\end{align}
	where the constant $C>0$ is independent of $\vep$ and $\tn$. 
	Thus if we  choose $\varrho$ and $\bar{M}$ such that 
	\begin{align}\label{delta1}
	2\sqrt{2C_0}C_3^{\f13}C_5^{\f23}\varrho^{\f13}\bar{M}^{\f13}(1+\bar{M})\leq\f14,
	\end{align} 
	then by  the fact $|\tn|^{-1}\leq\varrho$ for $n\neq0$, \eqref{L2-0} implies that for sufficiently small $\vep$, 
	\begin{align}\label{L2}
		\begin{aligned}
			\big\|\widehat{\FW}_n\big\|_{L^2} \lesssim|\tn|^{-\f23}|\log\vep|^{1+\frac{\eta}{3}}\left(\vep^{-\f14} \big\|Z^{\f12}{\FR}_{n}\big\|_{L^2}+\big\|{\FR}_{n}\big\|_{L^2}\right).
		\end{aligned}
	\end{align}
		Substituting \eqref{L2} into \eqref{2.6} and \eqref{2.12} respectively yields
	\begin{align}\label{L2-1}
	\sqrt{\vep}\left(\|\pa_y\widehat{\FW}_n\|_{L^2}+|\tn|\|\widehat{\FW}_n\|_{L^2}\right)\lesssim|\tn|^{-\f13} |\log\vep|^{1+\frac{\eta}{3}}\left(\vep^{-\f14} \big\|Z^{\f12}{\FR}_{n}\big\|_{L^2}+\big\|{\FR}_{n}\big\|_{L^2}\right),
	\end{align}
	and
	\begin{align}\label{L2-2}
	\|Z^{\f12}\widehat{\FW}_n\|_{L^2}\lesssim|\tn|^{-\f56}|\log\vep|^{1+\frac{\eta}{3}}\left( \big\|Z^{\f12}{\FR}_{n}\big\|_{L^2}+{\vep}^{\f14}\big\|{\FR}_{n}\big\|_{L^2}\right).
	\end{align}
	Combining \eqref{L2}-\eqref{L2-2} yields \eqref{2.18} and $\delta_1$ is determined by \eqref{delta1}. 
	And this completes  the proof of the lemma.
\end{proof}

Finally, in order to prove Proposition \ref{prop2.2} we need to obtain the weighted estimate $\big\|Z^{\f12}\pa_y\widehat{\FW}_n\big\|_{L^2}$. Similarly,  to avoid  the commutator of the weight
function and  the pressure $p$, we take curl on the first equality of \eqref{pr-new-n} to obtain
\begin{align}\label{curl-u}
\begin{aligned}
	&i\tn~ curl\left[(1+\frac{\mu}{\kappa})U_s\widehat{\FU}_n-G_s\widehat{\FH}_n\right]-\mu\vep(\pa_y^2-\tn^2)\omega_{u,n}\\
&=curl\left(\FR_{\FU,n}-i\tn\varepsilon^{\f12}\bm{A}_{\FU}\widehat{\FH}_n-\varepsilon^{\f12}\bm{B}_{\FU}\pa_{y}\widehat{\FH}_n-\bm{C}_{\FU}\widehat{\FH}_n-\varepsilon^{-\f12}\hat{\psi}_n\FD_{\FU}\right)\triangleq curl~ \widehat{\FR}_{\FU,n}.
\end{aligned}\end{align}
\begin{lemma}\label{prop2.7}
	For sufficient small $\vep$ and any $\eta>0$, there exists positive constant $C$, independent of $\vep$ and $n$, such that 
	\begin{align}\label{2.18-1}
	\begin{aligned}
	\sqrt{\vep}\left(\big\|Z^{\f12}\pa_y\widehat{\FW}_{n}\big\|_{L^2}+|\tn|~\big\|Z^{\f12}\widehat{\FW}_{n}\big\|_{L^2}\right)\leq~ &C|\log\vep|^{\f12+\f\eta 6}\left(|\tn|^{-\f12}\big\|Z^{\f12}{\FR}_n\big\|_{L^2}+|\tn|^{\f12}\big\|Z^{\f12}\widehat{\FW}_n\big\|_{L^2}\right)\\
	&+C{\vep}^{\f14}\left(\big\|\widehat{\FW}_n\big\|_{L^2}+\big\|\widehat{\FW}_n\big\|_{L^2}^{\f12}\big\|{\FR}_n\big\|_{L^2}^{\f12}\right).
	\end{aligned}\end{align}
\end{lemma}
\begin{proof}
	We take inner product of \eqref{curl-u} with $-Z\hat{\phi}_n$, and the second equation for $\widehat{\FH}$ in \eqref{pr-new-n} with $\displaystyle G_s(\frac{y}{\sqrt{\varepsilon}})Z\widehat{\FH}_n$ respectively, then 
	take the real part of its summation  to obtain
	\begin{align}\label{est-weightderi}
	\sum_{i=1}^{3}J_i=0,
	\end{align}
	where
		\begin{align*}\begin{aligned}
		&J_1=\tn~\mbox{Im}\left(\left\langle curl\left[(1+\frac{\mu}{\kappa})U_s\widehat{\FU}_n-G_s\widehat{\FH}_n\right],~Z\hat\phi_n\right\rangle+\left\langle \widehat{\FU}_n,~G_sZ\widehat{\FH}_n\right\rangle\right),\\
		&J_2=\vep\mbox{Re}\left(\left\langle\mu(\pa_y^2-\tn^2)\omega_{u,n},~Z\hat\phi_n\right\rangle-\left\langle\kappa(\pa_y^2-\tn^2)\widehat{\FH}_{n},~G_sZ\widehat{\FH}_n\right\rangle\right),\\
		&J_3=\mbox{Re}\left(\left\langle curl~\widehat{\FR}_{\FU,n},~Z\hat{\phi}_n\right\rangle-\left\langle\FR_{\FH,n}+2\kappa\varepsilon^{\f12}b_p\partial_y\widehat{\FH}_n-\bm{C}_{\FH}\widehat{\FH}_n-\varepsilon^{-\f12}\hat{\psi}_n\FD_{\FH},~G_sZ\widehat{\FH}_n\right\rangle\right).
		\end{aligned}
	\end{align*}
We estimate $J_i, i=1,2,3$ term by term. Firstly, for $J_1$ one has by integration by parts and the boundary condition $\hat{\phi}_n|_{y=0}=0$ that
	\begin{align*}
		&\left\langle curl\left[(1+\frac{\mu}{\kappa})U_s\widehat{\FU}_n-G_s\widehat{\FH}_n\right],~Z\hat\phi_n\right\rangle\\
		&=-\left\langle (1+\frac{\mu}{\kappa})U_s\widehat{\FU}_n-G_s\widehat{\FH}_n,~Z\widehat{\FU}_n\right\rangle-\left\langle (1+\frac{\mu}{\kappa})U_s\hat{u}_n-G_s\hat{h}_n,~Z_y\hat\phi_n\right\rangle.
	\end{align*}
	We apply the above equality to $J_1$ and get
	\begin{align*}
		J_1=-\tn~\mbox{Im}\left\langle (1+\frac{\mu}{\kappa})U_s\hat{u}_n-G_s\hat{h}_n,~Z_y\hat\phi_n\right\rangle,
	\end{align*}
	and by virtue of \eqref{w1_1} it implies that for sufficient small $\vep,$
	\begin{align}\label{est-J1}
	\begin{aligned}
		|J_1|&\lesssim|\tn||\log\vep|^{1+\f\eta 3} \big(\|Z^{\f12}\hat{u}_n\|_{L^2}+\|Z^{\f12}\hat{h}_n\|_{L^2}\big)\left(\|Z^{\f12}\hat{u}_n\|_{L^2}+\vep^{\f14+\delta}\|\hat{u}_n\|_{L^2}\right)\\
		&\lesssim |\tn||\log\vep|^{1+\f\eta 3}\big\|Z^{\f12}\widehat{\FW}_n\big\|_{L^2}^2+\sqrt{\vep}|\tn|\|\hat{u}_n\|_{L^2}^2.
	\end{aligned}\end{align}
	
	Secondly, as $\omega_{u,n}=(\pa_y^2-\tn^2)\hat{\phi}_n$, by integration by parts and the boundary conditions $Z|_{y=0}=\hat{\phi}_n|_{y=0}=0$,  we can write $J_2$ as
	\begin{align*}
	\begin{aligned}
	J_2=~&\vep\mbox{Re}\left(\mu\left\langle\omega_{u,n},~Z\omega_{u,n}\right\rangle+\kappa\left\langle\pa_y\widehat{\FH}_{n},~G_sZ\pa_y\widehat{\FH}_n\right\rangle+\kappa\tn^2\left\langle\widehat{\FH}_{n},~G_sZ\widehat{\FH}_n\right\rangle\right)\\
	&+\vep\mbox{Re}\left(\mu\left\langle\omega_{u,n},~2Z_y\hat{u}_n+Z_{yy}\hat\phi_n\right\rangle+\kappa\left\langle\pa_y\hat{h}_{n},~\pa_y(G_sZ)\hat{h}_n\right\rangle\right)\\
	\triangleq~&J_{2,1}+J_{2,2}.
	\end{aligned}\end{align*}
	From \eqref{A4}, it is easy to get
	\begin{align*}
		J_{2,1}&\geq\mu\vep\big\|Z^{\f12}\omega_{u,n}\big\|_{L^2}^2+\kappa\gamma_0\vep\left(\big\|Z^{\f12}\pa_y\widehat{\FH}_{n}\big\|_{L^2}^2+\tn^2\big\|Z^{\f12}\widehat{\FH}_{n}\big\|_{L^2}^2\right).
	\end{align*} 
	Similar to \eqref{useful},  it follows from $|\pa_y(G_sZ)|\leq C(1+\bar{M})$,  \eqref{z5} and the Hardy inequality that 
	\begin{align*}
\begin{aligned}
|J_{2,2}|\lesssim~&\vep\|\omega_{u,n}\|_{L^2}\left(\|Z_y\|_{L^\infty}\|\hat{u}_n\|_{L^2}+\|yZ_{yy}\|_{L^\infty}\|y^{-1}\hat{\phi}_n\|_{L^2}\right)+\vep\big\|\pa_y\widehat{\FH}_n\big\|_{L^2}\big\|\pa_y(G_sZ)\big\|_{L^\infty}\big\|\widehat{\FH}_n\big\|_{L^2}\\
\lesssim~&\vep\left(\big\|\pa_y\widehat{\FW}_n\big\|_{L^2}+|\tn|\big\|\widehat{\FW}_n\big\|_{L^2}\right)\big\|\widehat{\FW}_n\big\|_{L^2}.
\end{aligned}	\end{align*} 
	Combining the above three estimates yields
	\begin{align}\label{est-J2}
	\begin{aligned}
	J_2\geq~&\mu\vep\big\|Z^{\f12}\omega_{u,n}\big\|_{L^2}^2+\kappa\gamma_0\vep\left(\big\|Z^{\f12}\pa_y\widehat{\FH}_{n}\big\|_{L^2}^2+\tn^2\big\|Z^{\f12}\widehat{\FH}_{n}\big\|_{L^2}^2\right)\\
	&-C\vep\left(\big\|\pa_y\widehat{\FW}_n\big\|_{L^2}+|\tn|\big\|\widehat{\FW}_n\big\|_{L^2}\right)\big\|\widehat{\FW}_n\big\|_{L^2}.
	\end{aligned}\end{align}
	Next, for $J_3$, similar to \eqref{est-I4-3},  we  obtain
	\begin{align*}	\begin{aligned}
	\left|\left\langle curl~ \widehat{\FR}_{\FU,n}, ~Z{\hat{\phi}_n}\right\rangle\right|
	&\lesssim|\log\vep|^{1+\frac{\eta}{3}} \big\|Z^{\f12}{\FR}_{n}\big\|_{L^2}\left(\big\|Z^{\f12}\widehat{\FU}_n\big\|_{L^2}+\vep^{\frac{1}{4}+\delta}\big\|\hat{u}_n\big\|_{L^2}\right)\\
	&\quad+\vep\left(\big\|\pa_y\widehat{\FH}_n\big\|_{L^2}+|\tilde{n}|\big\|\widehat{\FH}_n\big\|_{L^2}\right)\big\|\widehat{\FU}_n\big\|_{L^2}+
	\sqrt{\vep}\big\|\widehat{\FW}_n\big\|_{L^2}^2.
	\end{aligned}
	\end{align*}
	As for \eqref{est-I4-1}, one has
	\begin{align*}
		\begin{aligned}
	 &\left|\left\langle\FR_{\FH,n}+2\kappa\varepsilon^{\f12}b_p\partial_y\widehat{\FH}_n-\bm{C}_{\FH}\widehat{\FH}_n-\varepsilon^{-\f12}\hat{\psi}_n\FD_{\FH},~G_sZ\widehat{\FH}_n\right\rangle\right|\\
	 &\lesssim \big\|Z^{\f12}{\FR}_{n}\big\|_{L^2}\big\|Z^{\f12}\widehat{\FH}_n\big\|_{L^2}+\vep\big\|\pa_y\widehat{\FH}_n\big\|_{L^2}\big\|\widehat{\FH}_n\big\|_{L^2}+\sqrt{\vep}\big\|\widehat{\FH}_n\big\|_{L^2}^2.
		\end{aligned}
	\end{align*}
	Combining the above two inequalities yields
	\begin{align}\label{est-J3}
		\begin{aligned}
	|J_3|\lesssim&|\log\vep|^{1+\frac{\eta}{3}} \big\|Z^{\f12}{\FR}_{n}\big\|_{L^2}\left(\big\|Z^{\f12}\widehat{\FW}_n\big\|_{L^2}+\vep^{\frac{1}{4}+\delta}\big\|\hat{u}_n\big\|_{L^2}\right)\\
	&+\vep\left(\big\|\pa_y\widehat{\FH}_n\big\|_{L^2}+|\tilde{n}|\big\|\widehat{\FH}_n\big\|_{L^2}\right)\big\|\widehat{\FW}_n\big\|_{L^2}+
	\sqrt{\vep}\big\|\widehat{\FW}_n\big\|_{L^2}^2\\
	\lesssim&|\tn||\log\vep|^{1+\f\eta 3}\big\|Z^{\f12}\widehat{\FW}_n\big\|_{L^2}^2+\sqrt{\vep}|\tn|\|\hat{u}_n\|_{L^2}^2+|\tn|^{-1}|\log\vep|^{1+\f\eta 3}\big\|Z^{\f12}{\FR}_n\big\|_{L^2}^2\\
	&+\vep\left(\big\|\pa_y\widehat{\FH}_n\big\|_{L^2}+|\tilde{n}|\big\|\widehat{\FH}_n\big\|_{L^2}\right)\big\|\widehat{\FW}_n\big\|_{L^2}+
	\sqrt{\vep}\big\|\widehat{\FW}_n\big\|_{L^2}^2,
		\end{aligned}
	\end{align}
	provided $\vep$ small enough.
	
	Thus,  we substitute \eqref{est-J1}, \eqref{est-J2} and \eqref{est-J3} into \eqref{est-weightderi} to obtain
	\begin{align*}
		\begin{aligned}
		&\vep\left(\big\|Z^{\f12}\omega_{u,n}\big\|_{L^2}^2+\big\|Z^{\f12}\pa_y\widehat{\FH}_{n}\big\|_{L^2}^2+\tn^2\big\|Z^{\f12}\widehat{\FH}_{n}\big\|_{L^2}^2\right)\\
		&\lesssim|\tn||\log\vep|^{1+\f\eta 3}\big\|Z^{\f12}\widehat{\FW}_n\big\|_{L^2}^2+\sqrt{\vep}|\tn|\|\hat{u}_n\|_{L^2}^2+|\tn|^{-1}|\log\vep|^{1+\f\eta 3}\big\|Z^{\f12}{\FR}_n\big\|_{L^2}^2\\
		&\quad+\vep\left(\big\|\pa_y\widehat{\FH}_n\big\|_{L^2}+|\tilde{n}|\big\|\widehat{\FH}_n\big\|_{L^2}\right)\big\|\widehat{\FW}_n\big\|_{L^2}+
		\sqrt{\vep}\big\|\widehat{\FW}_n\big\|_{L^2}^2.
		\end{aligned}\end{align*}
		Then, applying \eqref{2.6} and \eqref{2.9} to the above inequality yields
		\begin{align*}\begin{aligned}
		&\vep\left(\big\|Z^{\f12}\omega_{u,n}\big\|_{L^2}^2+\big\|Z^{\f12}\pa_y\widehat{\FH}_{n}\big\|_{L^2}^2+\tn^2\big\|Z^{\f12}\widehat{\FH}_{n}\big\|_{L^2}^2\right)\\
		&\lesssim|\tn|^{-1}|\log\vep|^{1+\f\eta 3}\big\|Z^{\f12}{\FR}_n\big\|_{L^2}^2+|\tn|~|\log\vep|^{1+\frac{\eta}{3}}\big\|Z^{\f12}\widehat{\FW}_n\big\|_{L^2}^2\\
		&\quad+
		\sqrt{\vep}\left(\big\|\widehat{\FW}_n\big\|_{L^2}^2+\big\|\widehat{\FW}_n\big\|_{L^2}\big\|{\FR}_n\big\|_{L^2}+\big\|\widehat{\FW}_n\big\|_{L^2}^{\f32}\big\|{\FR}_n\big\|_{L^2}^{\f12}\right)\\
		&\lesssim|\tn|^{-1}|\log\vep|^{1+\f\eta 3}\big\|Z^{\f12}{\FR}_n\big\|_{L^2}^2+|\tn|~|\log\vep|^{1+\frac{\eta}{3}}\big\|Z^{\f12}\widehat{\FW}_n\big\|_{L^2}^2+
		\sqrt{\vep}\left(\big\|\widehat{\FW}_n\big\|_{L^2}^2+\big\|\widehat{\FW}_n\big\|_{L^2}\big\|{\FR}_n\big\|_{L^2}\right).
		\end{aligned}
	\end{align*}
This and  Lemma \ref{lmw3} give \eqref{2.18-1}. And the proof of the lemma is completed.

\end{proof}

\subsection{Final estimates} Now we are ready to give the proof of Proposition \ref{prop2.2}.\\

\underline{\it Proof of Proposition \ref{prop2.2}:} Similar to \cite{GM}, the existence of solution $(\FU,\FH)$ to the problem \eqref{2.0} follows from a standard procedure: we can replace $-\mu\vep\Delta \FU$ and $-\ka\vep\Delta \FH$ by $-\mu\vep\Delta \FU+s\FU$ and $-\mu\vep\Delta \FH+s\FH$ respectively, with $s>0$. It is straightforward to show the existence for sufficiently large $s$. One can check that a priori estimate \eqref{prop2-1} is uniform in $s$. Therefore, the existence part follows from a standard continuity argument. We omit the detail for brevity. In what follows, we focus on the a priori estimate \eqref{prop2-1}. 
The proof is divided into two steps.\\

\underline{Step 1: $L^2$-estimate}. By \eqref{2.18} and \eqref{2.18-1}, we  obtain 
\begin{align}\label{est-weightgradient}
\sqrt{\vep}\left(\|Z^{\f12}\pa_y\widehat{\FW}_n\|_{L^2}+|\tn|\|Z^{\f12}\widehat{\FW}_n\|_{L^2}\right)\lesssim|\tn|^{-\f13} |\log\vep|^{\frac{3+\eta}{2}}\left( \big\|Z^{\f12}{\FR}_{n}\big\|_{L^2}+\vep^{\f14}\big\|{\FR}_{n}\big\|_{L^2}\right).
\end{align} 
Combining \eqref{2.18} with \eqref{est-weightgradient} yields
\begin{align*}
\begin{aligned}
&|\tn|^{\f23}\big\|\widehat{\FW}_n\big\|_{L^2}+\vep^{-\f14}|\tn|^{\f56}\big\|Z^{\f12}\widehat{\FW}_n\big\|_{L^2}+\sqrt{\vep}|\tn|^{\f13}\left(\|\pa_y\widehat{\FW}_n\|_{L^2}+|\tn|\big\|\widehat{\FW}_n\big\|_{L^2}\right)\\
&+{\vep}^{\f14}|\tn|^{\f13}\left(\|Z^{\f12}\pa_y\widehat{\FW}_n\|_{L^2}+|\tn|\big\|Z^{\f12}\widehat{\FW}_n\big\|_{L^2}\right)\\
&\lesssim |\log\vep|^{\frac{3+\eta}{2}}\left(\big\|{\FR}_{n}\big\|_{L^2}+\vep^{-\f14} \big\|Z^{\f12}{\FR}_{n}\big\|_{L^2}\right).
\end{aligned}
\end{align*}
Then by  Lemma \ref{lem2.3}, Lemma \ref{lm2.1} and the fact $|\tn|^{-1}\leq\varrho, n\neq0$, one has
\begin{align*}
\begin{aligned}
&\big\|{\FW}_n\big\|_{L^2}+\vep^{-\f14}\big\|Z^{\f12}{\FW}_n\big\|_{L^2}+\sqrt{\vep}\left(\|\pa_y{\FW}_n\|_{L^2}+|\tn|\big\|{\FW}_n\big\|_{L^2}\right)\\
&+{\vep}^{\f14}\left(\|Z^{\f12}\pa_y{\FW}_n\|_{L^2}+|\tn|\big\|Z^{\f12}{\FW}_n\big\|_{L^2}\right)\\
&\lesssim\big\|\widehat{\FW}_n\big\|_{L^2}+\vep^{-\f14}\big\|Z^{\f12}\widehat{\FW}_n\big\|_{L^2}+\sqrt{\vep}\left(\|\pa_y\widehat{\FW}_n\|_{L^2}+|\tn|\big\|\widehat{\FW}_n\big\|_{L^2}\right)\\
&\quad+{\vep}^{\f14}\left(\|Z^{\f12}\pa_y\widehat{\FW}_n\|_{L^2}+|\tn|\big\|Z^{\f12}\widehat{\FW}_n\big\|_{L^2}\right)\\
&\leq C|\log\vep|^{\frac{3+\eta}{2}}\left(\|(\Ff_n,\Fq_n)\|_{L^2}+\vep^{-\f14}\|Z^{\f12}(\Ff_n,\Fq_n)\|_{L^2}\right),
\end{aligned}
\end{align*}
where the positive constant $C$ is independent of $\vep$ and $n$. Therefore, by Paserval equality one has
\begin{align}\label{L2-all}
\begin{aligned}
&\vep^{-\f14}\|\CQ_0{\FW}\|_{L^2(\Omega)}+\vep^{-\f12}\|Z^{\f12}\CQ_0{\FW}\|_{L^2(\Omega)}+\vep^{\f14}\|\nabla\CQ_0{\FW}\|_{L^2(\Omega)}+\|Z^{\f12}\nabla\CQ_0{\FW}\|_{L^2(\Omega)}\\
&=\vep^{-\f14}\left\|\Big\{\|{\FW}_n\|_{L^2}\Big\}_{n\neq0}\right\|_{l^2}+\vep^{-\f12}\left\|\Big\{\|Z^{\f12}{\FW}_n\|_{L^2}\Big\}_{n\neq0}\right\|_{l^2}\\
&\quad+\vep^{\f14}\left\|\Big\{\|(\pa_y{\FW}_n, i\tn{\FW}_n)\|_{L^2}\Big\}_{n\neq0}\right\|_{l^2}+\left\|\Big\{\|Z^{\f12}(\pa_y{\FW}_n, i\tn{\FW}_n)\|_{L^2}\Big\}_{n\neq0}\right\|_{l^2}\\
&\leq C\vep^{-\f14}|\log\vep|^{\f{3+\eta}{2}}\left(\left\|\Big\{\|(\Ff_n,\Fq_n)\|_{L^2}\Big\}_{n\neq0}\right\|_{l^2}+\vep^{-\f14}\left\|\Big\{\|Z^{\f12}(\Ff_n,\Fq_n)\|_{L^2}\Big\}_{n\neq0}\right\|_{l^2}\right)\\
&=C\vep^{-\f14}|\log\vep|^{\f{3+\eta}{2}}\left(\|\CQ_0(\Ff,\Fq)\|_{L^2(\Omega)}+\vep^{-\f14}\|Z^{\f12}\CQ_0(\Ff,\Fq)\|_{L^2(\Omega)}\right).
\end{aligned}
\end{align} 

\underline{Step 2: $L^\infty$-estimate}. By using \eqref{lem2.3-1} with $p=\infty$,  the standard interpolation: $\|f\|_{L^\infty}\leq \sqrt{2}\|\pa_yf\|_{L^2}^{\f12}\|f\|_{L^2}^{\f12}$, and the estimate \eqref{2.18}, 
we divide the estimation on  $\|{\FW}_n\|_{L^\infty}$ into  two parts. Firstly,  for $1\leq |n|\leq \vep^{-1}$,
$$
\begin{aligned}
\|{\FW}_n\|_{L^\infty}&\lesssim\|\widehat{\FW}_n\|_{L^\infty}\leq \sqrt{2} |\tn|^{-\f12}\vep^{-\f14}\left(|\tn|^{\f13}\vep^{\f12}\|\pa_y\widehat{\FW}_n\|_{L^2}\right)^{\f12}\left(|\tn|^{\f23}\|\widehat{\FW}_n\|_{L^2}\right)^{\f12}\\
&\lesssim |n|^{-\f12}\vep^{-\f14}|\log\vep|^{1+\f{\eta}{3}}\left(\|\FR_n\|_{L^2}+\vep^{-\f14}\|Z^{\f12}\FR_n\|_{L^2}\right).
\end{aligned}
$$
Secondly,  for $|n|>\vep^{-1}$,
$$
\begin{aligned}
\|{\FW}_n\|_{L^\infty}&\lesssim\|\widehat{\FW}_n\|_{L^\infty}\leq \sqrt{2} |\tn|^{-\f56}\vep^{-\f12}\left(|\tn|^{\f13}\vep^{\f12}\|\pa_y\widehat{\FW}_n\|_{L^2}\right)^{\f12}\left(|\tn|^{\f43}\vep^{\f12}\|\widehat{\FW}_n\|_{L^2}\right)^{\f12}\\
&\lesssim |n|^{-\f{7}{12}}\vep^{-\f14}|\log\vep|^{1+\f{\eta}{3}}\left(\|\FR_n\|_{L^2}+\vep^{-\f14}\|Z^{\f12}\FR_n\|_{L^2}\right).
\end{aligned}
$$
Thus, from the above two inequalities, it follows that by Cauchy-Schwarz inequality, 
\begin{align}\label{L-infty}
\begin{aligned}
\sum_{n\neq 0}\|\FW_n\|_{L^\infty}\lesssim&\vep^{-\f14}|\log\vep|^{1+\f{\eta}{3}}\sum_{1\leq|n|\leq\vep^{-1}}|n|^{-\f{1}{2}}\left(\|\FR_n\|_{L^2}+\vep^{-\f14}\|Z^{\f12}\FR_n\|_{L^2}\right)\\
&+\vep^{-\f14}|\log\vep|^{1+\f{\eta}{3}}\sum_{|n|>\vep^{-1}}|n|^{-\f{7}{12}}\left(\|\FR_n\|_{L^2}+\vep^{-\f14}\|Z^{\f12}\FR_n\|_{L^2}\right)\\
\lesssim&\vep^{-\f14}|\log\vep|^{1+\f{\eta}{3}}\left(\left\|\Big\{\|\FR_n\|_{L^2}\Big\}_{n\neq0}\right\|_{l^2}+\vep^{-\f14}\left\|\Big\{\|Z^{\f12}\FR_n\|_{L^2}\Big\}_{n\neq0}\right\|_{l^2}\right)\\
&\cdot\left[\left(\sum_{1\leq|n|\leq\vep^{-1}}|n|^{-1}\right)^{\f12}+\left(\sum_{|n|>\vep^{-1}}|n|^{-\f{7}{6}}\right)^{\f12}\right]\\
\leq& C \vep^{-\f14}|\log\vep|^{\f32+\f\eta{3}}\left(\|\CQ_0(\Ff,\Fq)\|_{L^2(\Omega)}+\vep^{-\f14}\|Z^{\f12}\CQ_0(\Ff,\Fq)\|_{L^2(\Omega)}\right),
\end{aligned}\end{align}
where we have used the third inequality in \eqref{R} and the Paserval equality in the last inequality. 
Finally, it is easy to obtain the desired estimate \eqref{prop2-1} from Lemma \ref{prop2.1}, \eqref{L2-all} and \eqref{L-infty}. 
The proof of Proposition \ref{prop2.2} is completed. \qed

\section{Nonlinear stability}
Recall the solution space $\CX$ defined in \eqref{X}. For any $(\Fq,\Fr)\in\CX$, we define the nonlinear map $\Phi(\Fq,\Fr)=(\widetilde{\FU},\widetilde{\FH})$ as the solution to the following linear problem: 
\begin{equation}
\left\{
\begin{aligned}\label{3.1}
&U_s\pa_x\widetilde{\FU}+\tilde{v}\pa_yU_s\Fe_1-H_s\pa_x\widetilde{\FH}-\tilde{g}\pa_yH_s\Fe_1+ \nabla P-\mu\vep\Delta\widetilde{\FU}=-\Fq\cdot\nabla \Fq+\Fr\cdot\nabla \Fr+\Ff_\FU,\\
&U_s\pa_x\widetilde{\FH}+\tilde{v}\pa_yH_s\Fe_1-H_s\pa_x\widetilde{\FU}-\tilde{g}\pa_yU_s\Fe_1-\ka\vep\Delta\widetilde{\FH}=-\Fq\cdot\nabla \Fr+\Fr\cdot\nabla \Fq+\Ff_\FH,\\
&\nabla\cdot \widetilde{\FU}=\nabla\cdot \widetilde{\FH}=0,\\
&\widetilde{\FU}|_{y=0}=(\pa_y\tilde{h},\tilde{g})|_{y=0}=\mathbf{0}.
\end{aligned}
\right.
\end{equation}
The existence of solution operator $\Phi$ in $\CX$ is guaranteed by Proposition \ref{prop2.2} provided that the source term $$(\tilde{\Ff}_\FU, \tilde{\Ff}_\FH)
\triangleq(-\Fq\cdot\nabla \Fq+\Fr\cdot\nabla \Fr+\Ff_\FU,~ -\Fq\cdot\nabla \Fr+\Fr\cdot\nabla \Fq+\Ff_\FH)$$ satisfies the compatibility conditions \eqref{ass-q} and \eqref{2.3}. Then the proof of main result follows from showing the contractiveness of $\Phi$ in a suitable domain of $\CX$ provided that the external force $(\Ff_{\FU},\Ff_{\FH})$ is suitably small. Consequently, it remains to verify \eqref{ass-q} and \eqref{2.3} for $(\tilde{\Ff}_\FU, \tilde{\Ff}_\FH)$, and to prove $\Phi$ is a contraction map. 

For $(\Fq,\Fr)\in\CX$, direct calculation shows that $(\tilde{\Ff}_\FU, \tilde{\Ff}_\FH)$ satisfies \eqref{ass-q}. To show \eqref{2.3} for $(\tilde{\Ff}_\FU, \tilde{\Ff}_\FH)$, let us recall the projections on the zeroth Fourier mode $\CP_0$ and on non-zero Fourier mode $\CQ_0$. Let $\Fs=(s_1,s_2)$ and $\Ft=(t_1,t_2)$ be any two divergence-free vectors satisfying boundary condition $s_2|_{y=0}=t_2|_{y=0}=0$. Then $\CP_0\Fs$ and $ \CP_0\Ft$ depend only on $y$ and $\CP_0s_2=\CP_0t_2=0$, which implies that
\begin{align}
\Fs\cdot\nabla\Ft&=\CP_0\Fs\cdot\nabla\CP_0\Ft+\CP_0\Fs\cdot\nabla\CQ_0\Ft+\CQ_0\Fs\cdot\nabla\CP_0\Ft+\CQ_0\Fs\cdot\nabla\CQ_0\Ft\nonumber\\
&=\CP_0s_1\pa_x\CQ_{0}\Ft+\CQ_0s_2(\pa_y\CP_0t_1)\Fe_1+\CQ_0\Fs\cdot\nabla\CQ_0\Ft.\label{3.4}
\end{align}
 Observe that
$$\begin{aligned}
\CP_0(\Fs\cdot\nabla\Ft)&=\CP_0\left(\CQ_0\Fs\cdot\nabla\CQ_0\Ft\right)=\pa_y\CP_0\left(\CQ_0s_2\CQ_0\Ft\right)+\CP_0\left(\CQ_0s_1\pa_x\CQ_0\Ft-\pa_y\CQ_0s_2\CQ_0\Ft\right)\\
&=\pa_y\CP_0\left(\CQ_0s_2\CQ_0\Ft\right)+\CP_0\left(\sum_{n\neq0,m\neq0}e^{i(\tn+\tilde{m})x}\left(s_{1,n}i\tilde{m}-\pa_ys_{2,n}\right)\Ft_{m}\right)\\
&=\pa_y\CP_0\left(\CQ_0s_2\CQ_0\Ft\right)-\sum_{n\neq0}(i\tilde{n}s_{1,n}+\pa_ys_{2,n})\Ft_{-n}=\pa_y\CP_0\left(\CQ_0s_2\CQ_0\Ft\right).
\end{aligned}
$$
Applying the above equality to $(\tilde{\Ff}_\FU, \tilde{\Ff}_\FH)$ yields
\begin{align*}
	\left(\CP_0\tilde{\Ff}_\FU, \CP_0\tilde{\Ff}_\FH\right)=\Big(\pa_y\CP_0\big(-\CQ_{0}q_2\CQ_{0}\Fq+\CQ_{0}r_2\CQ_{0}\Fr\big),~\pa_y\CP_0\big(-\CQ_{0}q_2\CQ_{0}\Fr+\CQ_{0}r_2\CQ_{0}\Fq\big)\Big),
\end{align*}
and then
\begin{align*}
	\left(\CI\CP_0\tilde{\Ff}_\FU,~ \pa_y^{-1}\CP_0\tilde{\Ff}_\FH\right)=\Big(\CP_0\big(-\CQ_{0}q_2\CQ_{0}\Fq+\CQ_{0}r_2\CQ_{0}\Fr\big),~\CP_0\big(-\CQ_{0}q_2\CQ_{0}\Fr+\CQ_{0}r_2\CQ_{0}\Fq\big)\Big).
\end{align*}
Since $|\CP_0f|\lesssim\|f\|_{L^1(\T_\varrho)}$, by  Cauchy-Schwarz inequality, 
it follows that
\begin{align}\label{L1-P0}
	\begin{aligned}
	\left\|\left(\CI\CP_0\tilde{\Ff}_\FU,~ \pa_y^{-1}\CP_0\tilde{\Ff}_\FH\right)\right\|_{L^1(\R_+)}
	&\lesssim\left\|-\CQ_{0}q_2\CQ_{0}\Fq+\CQ_{0}r_2\CQ_{0}\Fr\right\|_{L^1(\Omega)}+\left\|-\CQ_{0}q_2\CQ_{0}\Fr+\CQ_{0}r_2\CQ_{0}\Fq\right\|_{L^1(\Omega)}\\
	&\lesssim\left\|\CQ_{0}(\Fq,\Fr)\right\|_{L^2(\Omega)}^2\lesssim{\vep}^{\f12}\left\|(\Fq,\Fr)\right\|_{\CX}^2,
	\end{aligned}
\end{align}
and by Paserval equality,
\begin{align}\label{L2-P0}
\begin{aligned}
		\left\|\left(\CI\CP_0\tilde{\Ff}_\FU,~ \pa_y^{-1}\CP_0\tilde{\Ff}_\FH\right)\right\|_{L^2(\R_+)}
		&\leq\left\|-\CQ_{0}q_2\CQ_{0}\Fq+\CQ_{0}r_2\CQ_{0}\Fr\right\|_{L^2(\Omega)}+\left\|-\CQ_{0}q_2\CQ_{0}\Fr+\CQ_{0}r_2\CQ_{0}\Fq\right\|_{L^2(\Omega)}\\
	&\lesssim\left\|\CQ_{0}(\Fq,\Fr)\right\|_{L^\infty(\Omega)}\left\|\CQ_{0}(\Fq,\Fr)\right\|_{L^2(\Omega)}\lesssim\vep^{\f14}\left\|(\Fq,\Fr)\right\|_{\CX}^2,
\end{aligned}\end{align}
where we have used the fact  that $\|\CQ_0f\|_{L^\infty(\Omega)}\leq C\sum_{n\neq0}\|f_n\|_{L^\infty(\R_+)}.$
Thus we verify the first part of \eqref{2.3}. Moreover, by the commutativity of the weight $Z^{\f12}$ and the projection operators $\CP_0, \CQ_0$, similar to \eqref{L2-P0},  it holds
\begin{align}\label{w-L2-P0}
\begin{aligned}
&\left\|Z^{\f12}\left(\CI\CP_0\tilde{\Ff}_\FU,~ \pa_y^{-1}\CP_0\tilde{\Ff}_\FH\right)\right\|_{L^2(\R_+)}\lesssim\left\|\CQ_{0}(\Fq,\Fr)\right\|_{L^\infty(\Omega)}\left\|Z^{\f12}\CQ_{0}(\Fq,\Fr)\right\|_{L^2(\Omega)}\lesssim{\vep}^{\f12}\left\|(\Fq,\Fr)\right\|_{\CX}^2.
\end{aligned}\end{align}
 For the second part of \eqref{2.3}, we use \eqref{3.4} to have
\begin{align*}
\begin{aligned}
\|\CQ_0(\Fs\cdot\nabla\Ft)\|_{L^2(\Omega)}&\leq\|\Fs\cdot\nabla\Ft\|_{L^2(\Omega)}\\
&\lesssim \left(\|\CP_0s_1\|_{L^\infty(\R_+)}+\|\CQ_0\Fs\|_{L^\infty(\Omega)}\right)\|\nabla\CQ_{0}\Ft\|_{L^2(\Omega)}+\|\CQ_{0}s_2\|_{L^\infty(\Omega)}\|\pa_y\CP_0t_1\|_{L^2(\R_+)}.
\end{aligned}\end{align*}
Apply the above inequality to $\CQ_{0}(\tilde{\Ff}_\FU, \tilde{\Ff}_\FH)$ and obtain
\begin{align}\label{L2-Q0}
	\begin{aligned}
\left\|\CQ_{0}(\tilde{\Ff}_\FU, \tilde{\Ff}_\FH)\right\|_{L^2(\Omega)}\lesssim&\Big(\|\CP_0(\Fq,\Fr)\|_{L^\infty(\R_+)}+\|\CQ_0(\Fq,\Fr)\|_{L^\infty(\Omega)}\Big)\|\nabla\CQ_{0}(\Fq,\Fr)\|_{L^2(\Omega)}\\
&+\|\CQ_{0}(\Fq,\Fr)\|_{L^\infty(\Omega)}\|\pa_y\CP_0(\Fq,\Fr)\|_{L^2(\R_+)}+\left\|\CQ_{0}(\Ff_\FU, \Ff_\FH)\right\|_{L^2(\Omega)}\\
\lesssim&\vep^{-\f14}\left\|(\Fq,\Fr)\right\|_{\CX}^2+\left\|(\Ff_\FU, \Ff_\FH)\right\|_{L^2(\Omega)},
	\end{aligned}
\end{align}
which implies the second part of \eqref{2.3}. Moreover, we can show that 
\begin{align}\label{w-L2-Q0}
\begin{aligned}
\left\|Z^{\f12}\CQ_{0}(\tilde{\Ff}_\FU, \tilde{\Ff}_\FH)\right\|_{L^2(\Omega)}\lesssim&\Big(\|\CP_0(\Fq,\Fr)\|_{L^\infty(\R_+)}+\|\CQ_0(\Fq,\Fr)\|_{L^\infty(\Omega)}\Big)\|Z^{\f12}\nabla\CQ_{0}(\Fq,\Fr)\|_{L^2(\Omega)}\\
&+\|\CQ_{0}(\Fq,\Fr)\|_{L^\infty(\Omega)}\|Z^{\f12}\pa_y\CP_0(\Fq,\Fr)\|_{L^2(\R_+)}+\left\|Z^{\f12}\CQ_{0}(\Ff_\FU, \Ff_\FH)\right\|_{L^2(\Omega)}\\
\lesssim&\left\|(\Fq,\Fr)\right\|_{\CX}^2+\left\|Z^{\f12}(\Ff_\FU, \Ff_\FH)\right\|_{L^2(\Omega)}.
\end{aligned}
\end{align}

Next, we apply Proposition \ref{prop2.2} to the problem \eqref{3.1}, and obtain
	\begin{align*}
	\begin{aligned}
	&\left\|\Phi(\Fq,\Fr)\right\|_{\CX}=\|(\widetilde{\FU},\widetilde{\FH})\|_{\CX}\\
	&\lesssim\vep^{-1}\left(\left\|\left(\CI\CP_0\tilde{\Ff}_\FU,~ \pa_y^{-1}\CP_0\tilde{\Ff}_\FH\right)\right\|_{L^1}+\vep^{\f14}\left\|\left(\CI\CP_0\tilde{\Ff}_\FU,~ \pa_y^{-1}\CP_0\tilde{\Ff}_\FH\right)\right\|_{L^2}+\big\|Z^{\f12}\left(\CI\CP_0\tilde{\Ff}_\FU,~ \pa_y^{-1}\CP_0\tilde{\Ff}_\FH\right)\big\|_{L^2}\right)\\
	&\quad+\vep^{-\f14}|\log\vep|^{\frac{3+\eta}{2}}\left(\|\CQ_0(\tilde{\Ff}_\FU,\tilde{\Ff}_\FH)\|_{L^2(\Omega)}+\vep^{-\f14}\|Z^{\f12}\CQ_0(\tilde{\Ff}_\FU,\tilde{\Ff}_\FH)\|_{L^2(\Omega)}\right).
	\end{aligned}\end{align*}
Combining the above inequality and the estimates \eqref{L1-P0}-\eqref{w-L2-Q0} yields
\begin{align}\label{est-map}
\begin{aligned}
	\left\|\Phi(\Fq,\Fr)\right\|_{\CX}
	&\lesssim\vep^{-\f12}\left\|(\Fq,\Fr)\right\|_{\CX}^2+\vep^{-\f14}|\log\vep|^{\frac{3+\eta}{2}}\left(\vep^{-\f14}\left\|(\Fq,\Fr)\right\|_{\CX}^2+\left\|(\Ff_\FU, \Ff_\FH)\right\|_{L^2(\Omega)}+\vep^{-\f14}\left\|Z^{\f12}(\Ff_\FU, \Ff_\FH)\right\|_{L^2(\Omega)}\right)\\
	&\leq C\vep^{-\f12}|\log\vep|^{\frac{3+\eta}{2}}\left\|(\Fq,\Fr)\right\|_{\CX}^2+C\vep^{-\f14}|\log\vep|^{\frac{3+\eta}{2}}\left(\left\|(\Ff_\FU, \Ff_\FH)\right\|_{L^2(\Omega)}+\vep^{-\f14}\left\|Z^{\f12}(\Ff_\FU, \Ff_\FH)\right\|_{L^2(\Omega)}\right),
\end{aligned}\end{align}
where the constant $C>0$ is independent of $\vep$. Therefore, \eqref{est-map} shows that the map $\Phi$ is well-defined from $\CX$ to $\CX$.
Moreover, by a similar  argument as above we can show that for any two vectors $(\Fq_1,\Fr_1), (\Fq_2,\Fr_2)\in\CX$, it holds
$$\|\Phi(\Fq_1-\Fq_2,\Fr_1-\Fr_2)\|_{\CX}
\leq C\vep^{-\f12}|\log\vep|^{\f{3+\eta}{2}}\big(\|(\Fq_1,\Fr_1)\|_{\CX}+\|(\Fq_2,\Fr_2)\|_{\CX}\big)\big\|(\Fq_1-\Fq_2,\Fr_1-\Fr_2)\big\|_{\CX} .
$$
Now, we are able to choose suitable $(\Ff_{\FU},\Ff_{\FH})$ to establish the contractiveness of map $\Phi$ in a suitable domain of $\CX$. Indeed, for any fixed $0<\alpha<1$, let $$\displaystyle
\|(\Ff_{\FU},\Ff_{\FH})\|_{L^2}+\vep^{-\f14}\|Z^{\f12}(\Ff_\FU,\Ff_{\FH})\|_{L^2(\Omega)}\leq\delta_2\vep^{\f34}|\log\vep|^{-(3+\eta)}\quad\mbox{with}\quad \delta_2=\frac{\alpha(2-\alpha)}{4C^2}, 
$$ 
and we consider the domain of $\CX$:
$$
\mathcal{D}:=\left\{(\widetilde{\FU},\widetilde{\FH})\in\CX~\bigg|~\big\|(\widetilde{\FU},\widetilde{\FH})\big\|_{\CX}\leq \f{\alpha\vep^{\f12}}{2C|\log\vep|^{\frac{3+\eta}{2}}}\right\}. 
$$
It is straightforward to check that $\Phi$ is a contraction map from $\mathcal{D}$ to $\mathcal{D}$. Therefore, the existence and uniqueness of the solution to \eqref{1.2} follow from the fixed point theorem. In addition,  the solution $(\widetilde{\FU},\widetilde{\FH})$ satisfies 
$$\|(\widetilde{\FU},\widetilde{\FH})\|_{\CX}\leq\frac{2C}{2-\alpha}\vep^{-\f14}|\log\vep|^{\frac{3+\eta}{2}}\left(\|(\Ff_\FU,\Ff_{\FH})\|_{L^2(\Omega)}+\vep^{-\f14}\|Z^{\f12}(\Ff_\FU,\Ff_{\FH})\|_{L^2(\Omega)}\right).
$$
That is,  we obtain \eqref{est}. Finally, it is easy to see that $-\mu\vep\Delta\widetilde{\FU}+\nabla P\in L^2(\Omega)$ and $-\ka\vep\Delta\widetilde{\FH}\in L^2(\Omega)$. Then by ellipticity of Stokes operators and Laplacian operators, we have
$\nabla^2\widetilde{\FU},\nabla^2\widetilde{\FH}, \nabla P\in L^2(\Omega).$ Therefore, the proof of Theorem \ref{thm1} is completed.\qed\\
\section{Appendix}
We give the detailed proof of Lemma \ref{lem2.3} as follows.\\

\underline{\it Proof of Lemma \ref{lem2.3}:} We focus on the estimates on $\hh$ since other components can be treated in a similar way. Recall that in \eqref{n-unknown}
\begin{align}\label{h_n} \hh=\pa_y\left(\frac{\psi_n}{H_s}\right)=\f1{H_s}\left(h_n-\vep^{-\f12}b_p\psi_n\right).\end{align}
It follows by  the Hardy inequality that
\begin{align}
\|\hh\|_{L^p}\leq C\|h_n\|_{L^p}+ C\|Yb_p\|_{L^\infty_Y}\|y^{-1}\psi_n\|_{L^p}\leq C(1+\bar{M})\|h_n\|_{L^p}.\label{lem2.3-5}
\end{align}
For the weighted $L^2$-norm, one has by using \eqref{z1} that 
\begin{align}\label{lem2.3-6}
\begin{aligned}
\|Z^{\f12}\hh\|_{L^2}&\leq C\|Z^{\f12}h_n\|_{L^2}+ C\vep^{\f14}\left\|\sqrt{\frac{Z(y)}{y}}\right\|_{L^\infty}\|Y^{\f32}b_p\|_{L^\infty_Y}\|y^{-1}\psi_n\|_{L^2}\\
&\leq C\|Z^{\f12}h_n\|_{L^2}+C\bar{M}\vep^{\f14}\|h_n\|_{L^2},
\end{aligned}\end{align}
and by virtue of $g_n=-i\tn\psi_n$,
\begin{align}\label{lem2.3-7}
\begin{aligned}
\vep^{\f14}|\tilde{n}|\|Z^{\f12}\hh\|_{L^2}&\leq C\vep^{\f14}|\tilde{n}|\|Z^{\f12}h_n\|_{L^2}+ C\left\|\sqrt{\frac{Z(y)}{y}}\right\|_{L^\infty}\|Y^{\f12}b_p\|_{L^\infty_Y}\|g_n\|_{L^2}\\
&\leq C\vep^{\f14}|\tilde{n}|\|Z^{\f12}h_n\|_{L^2}+C\bar{M}\|g_n\|_{L^2}.
\end{aligned}\end{align}
Then, taking the $y-$derivative in \eqref{h_n} gives
\begin{align}\label{h_n-y}
\pa_y\hh=\f{1}{H_s}\left(\pa_yh_n-2\vep^{-\f12}b_ph_n+\vep^{-1}(b_p^2-\pa_Yb_p)\psi_n\right).
\end{align}
It follows
\begin{align}\label{equiv-1}
\begin{aligned}\vep^{\f12}\|\pa_y\hh\|_{L^2}&\leq C\vep^{\f12}\|\pa_yh_n\|_{L^2}+C\|b_p\|_{L^\infty_Y}\|h_n\|_{L^2}+C\big\|Y(b_p^2-\pa_Yb_p)\big\|_{L^\infty_Y}\|y^{-1}\psi_n\|_{L^2}\\
&\leq C\vep^{\f12}\|\pa_yh_n\|_{L^2}+C(1+\bar{M}^2)\|h_n\|_{L^2},	
\end{aligned}
\end{align}
and 
\begin{align}\label{equiv-2}
\begin{aligned}
&\vep^{\f14}\|Z^{\f12}\pa_y\hh\|_{L^2}\\
&\leq C\vep^{\f14}\|Z^{\f12}\pa_yh_n\|_{L^2}+C\left\|\sqrt{\frac{Z(y)}{y}}\right\|_{L^\infty}\left[\|Y^{\f12}b_p\|_{L^\infty_Y}\|h_n\|_{L^2}
+\big\|Y^{\f32}(b_p^2-\pa_Yb_p)\big\|_{L^\infty_Y}\|y^{-1}\psi_n\|_{L^2}\right]\\
&\leq C\vep^{\f14}\|Z^{\f12}\pa_yh_n\|_{L^2}+C(1+\bar{M}^2)\|h_n\|_{L^2}.
\end{aligned}
\end{align}
In conclusion, we combine \eqref{lem2.3-5}, \eqref{lem2.3-6}, \eqref{lem2.3-7}, \eqref{equiv-1} and \eqref{equiv-2} to get
\begin{align}\label{equiv1}
\begin{aligned}
&\|\hat{h}_n\|_{L^p}\lesssim_{\bar{M}} \|h_n\|_{L^p},\quad 1<p\leq\infty,\\
&\|Z^{\f12}\hat{h}_n\|_{L^2}+\varepsilon^{\f14}\|\hat{h}_n\|_{L^2}\lesssim_{\bar{M}}\left(\|Z^{\f12}h_n\|_{L^2}+\varepsilon^{\f14}\|h_n\|_{L^2}\right),\\
&\|\hat{h}_n\|_{L^2}+\vep^{\f12}\|\pa_y\hat{h}_n\|_{L^2}\lesssim_{\bar{M}}\left(\|h_n\|_{L^2}+\vep^{\f12}\|\pa_yh_n\|_{L^2}\right),\\
&\|\hat{h}_n\|_{L^2}+\vep^{\f14}\|Z^{\f12}(\pa_y\hat{h}_n,\tn \hat{h}_n)\|_{L^2}\lesssim_{\bar{M}}\left(\|(h_n,g_n)\|_{L^2}+\vep^{\f14}\|Z^{\f12}(\pa_yh_n,\tn h_n)\|_{L^2}\right).
\end{aligned}
\end{align}

On the other hand, we can express $h_n$ in terms of $\hh$. Indeed, from \eqref{h_n} one has $\psi_n=H_s\pa_y^{-1}\hat{h}_n$, and then
\begin{align}\label{hh_n}
h_n=H_s\left(\hh+\vep^{-\f12}b_p\pa_y^{-1}\hat{h}_n\right),\quad \pa_yh_n=H_s\left(\pa_yh_n+2\vep^{-\f12}b_ph_n+\vep^{-1}(b_p^2+\pa_Yb_p)\psi_n\right).
\end{align}
Comparing \eqref{h_n}, \eqref{h_n-y} with \eqref{hh_n} and noting $\psi_n=\pa_y^{-1}h_n$, we use 
a similar  argument as above to obtain
\begin{align}\label{equiv2}
\begin{aligned}
&\|h_n\|_{L^p}\lesssim_{\bar{M}}\|\hat{h}_n\|_{L^p},\quad 1<p\leq\infty,\\
&\|Z^{\f12}h_n\|_{L^2}+\varepsilon^{\f14}\|h_n\|_{L^2}\lesssim_{\bar{M}}\left(\|Z^{\f12}\hat{h}_n\|_{L^2}+\varepsilon^{\f14}\|\hat{h}_n\|_{L^2}\right),\\
&\|h_n\|_{L^2}+\vep^{\f12}\|\pa_yh_n\|_{L^2}\lesssim_{\bar{M}}\left(\|\hat{h}_n\|_{L^2}+\vep^{\f12}\|\pa_y\hat{h}_n\|_{L^2}\right),\\
&\|h_n\|_{L^2}+\vep^{\f14}\|Z^{\f12}(\pa_yh_n,\tn h_n)\|_{L^2}\lesssim_{\bar{M}}\left(\|(\hat{h}_n,\hat{g}_n)\|_{L^2}+\vep^{\f14}\|Z^{\f12}(\pa_y\hat{h}_n,\tn \hat{h}_n)\|_{L^2}\right).
\end{aligned}
\end{align}
Combining \eqref{equiv1} with \eqref{equiv2}, we complete the proof of the lemma.  \qed

~\\

\noindent{\bf Acknowledgment:} 
The research of C.-J. Liu was supported by the National Key R\&D Program of China (No.2020YFA0712000), National Natural Science Foundation of China (No.11743009, 11801364), the Strategic Priority Research Program of Chinese Academy of Sciences (No.XDA25010402), Shanghai Sailing Program (No.18YF1411700), a startup grant from Shanghai Jiao Tong University (No.WF220441906), and the Program for Professor of Special Appointment (Eastern Scholar) at Shanghai Institutions of Higher Learning.
The research of T. Yang  was supported by the General Research Fund of Hong Kong CityU No 11302518 and the  Fundamental Research Funds for the Central Universities No.2019CDJCYJ001. The first and second authors'
research was also  supported by Hong Kong Institute for Advanced Study No.9360157.

\end{document}